\documentclass[12pt]{article}

\usepackage{lipsum}
\usepackage{amsfonts}
\usepackage{graphicx}
\usepackage{epstopdf}
\usepackage{algorithmic}
\usepackage{booktabs}
\usepackage{graphicx,epstopdf} 
\usepackage[caption=false]{subfig} 
\usepackage{enumitem}
\usepackage{hyperref}
\usepackage{amsmath,amssymb,amsthm}
\usepackage{mathtools}
\usepackage{xcolor}
\usepackage{lineno}
\usepackage[
paper=a4paper,
total={150mm,247mm},
includehead,      
includefoot,      
headsep=15pt,     
centering         
]{geometry}
\usepackage{fancyhdr}

\allowdisplaybreaks

\ifpdf
  \DeclareGraphicsExtensions{.eps,.pdf,.png,.jpg}
\else
  \DeclareGraphicsExtensions{.eps}
\fi

\theoremstyle{plain}             
\newtheorem{theorem}{Theorem}[section]

\newtheorem{proposition}[theorem]{Proposition}
\newtheorem{corollary}[theorem]{Corollary}

\theoremstyle{remark}            
\newtheorem{remark}[theorem]{Remark}

\title{Convergence of generalized cross-validation with applications to ill-posed
integral equations\thanks{Submitted to the editors DATE.}
\thanks{Funded  by  the  Deutsche  Forschungsgemeinschaft under Germany's
Excellence Strategy  - EXC-
2046/1, Projekt-ID 390685689 (The Berlin Mathematics Research Center MATH+, Project AA5-11)}}

\author{
	Tim Jahn\thanks{Institut f\"ur Mathematik, Technische Universit\"at Berlin,
	Germany (\texttt{jahn@tu-berlin.de}).}\and
    Mikhail Kirilin\thanks{Institut f\"ur Mathematik, Technische Universit\"at
    Berlin, Germany (\texttt{mkirilin@tu-berlin.de}).}
}
\date{}

\usepackage{amsopn}

\newcommand{\R}{\mathbb{R}}
\newcommand{\N}{\mathbb{N}}

\newcommand{\E}{\mathbb{E}}

\newcommand{\omd}{\chi_{\Omega_{t}}}

\newcommand{\omk}{\chi_{\Omega_t}}
\newcommand{\omkp}{\chi_{\Omega_t'}}

\ifpdf
\hypersetup {
	pdftitle={Convergence of generalized cross-validation with applications to
	ill-posed integral equations},
	pdfauthor={ T. Jahn, M. Kirilin}
}
\fi

\begin{document}

	\maketitle

\begin{abstract}

In this article, we rigorously establish the consistency of generalized
cross-validation as a parameter-choice rule for solving inverse problems. We
prove that the index chosen by leave-one-out GCV achieves a non-asymptotic,
order-optimal error bound
with high probability for polynomially ill-posed compact operators. Hereby it is
remarkable that the unknown true solution need not satisfy a
self-similarity condition, which is generally needed for other heuristic
parameter choice rules. We quantify the rate and demonstrate convergence
numerically on integral equation test cases, including image deblurring and CT
reconstruction.


\end{abstract}

\noindent{\it Keywords\/}:
statistical inverse problems, generalized cross-validation, consistency, error
estimates

\section{Introduction}\label{s1}

Generalized cross-validation (GCV) is a widely adopted parameter-selection
criterion for the regularized solution of ill-posed inverse problems. It is based
on splitting the data set into two parts: the first part is used to construct
a solution candidate for the task, while the second part is used to validate the
performance of the candidate. For a classic reference, see Stone
\cite{stone1974cross}. For a more recent one, see Hastie et al.
\cite{hastie2009elements} and Arlot \& Celisse \cite{arlot2010survey}. The
generalized cross-validation technique analyzed in this study is rooted in the
seminal work of Wahba and Craven \cite{craven1978smoothing}. They employed this
technique for the purpose of spline smoothing of noisy point evaluations of a
function. A salient feature of the rule is that it does not presuppose knowledge
of the noise level or the smoothness of the unknown function. In its original
form, the leave-one-out  method involves the estimation of a spline fit
for
all data points except for one, with the error of the unconsidered datum serving
as the quality criterion. There one varies a so-called smoothing parameter to
balance how well the candidate fits the data points with the norm of the
candidate. This ultimately results in a minimization problem over the smoothing
parameter. In the context of ill-posed integral equations, this
method has been employed for the purpose of selecting the regularization
parameter by Wahba \cite{wahba1977practical}, Vogel \cite{vogel1986optimal},
Lukas \cite{lukas1993asymptotic} and others. The scope of applications of GCV
and its variants have expanded, solidifying their position as prominent
 methods in the fields of high-dimensional statistics, data science,
and machine learning (see Witten \& Frank \cite{witten2002data}; Kuhn \&
Johnson\cite{kuhn2013applied}; Giraud\cite{giraud2021introduction}). Given the
significance of GCV as a practical rule in these areas, the present article
clarifies the theoretical properties of the original method.

In general, two types of convergence results for cross-validation are
distinguished. The vast majority is of weak type. There, the focus is on the
properties of the minimizer of the population counterpart of the random
data-driven functional rather than on the properties of the minimizer of the
(random) data-driven functional itself. While convergence results for minimizers
of the expected value offer valuable insight into the problem, from a statistical
perspective, they do not guarantee consistency of the original method. To date,
there are no substantial convergence results for GCV in the context of ill-posed
integral equations. Given the inherent instability of ill-posed integral
equations, this gap is unsatisfactory. The primary contribution of this
manuscript is an analysis of the convergence of GCV for ill-posed inverse
problems of strong type. Here, the properties of the minimizer of the
random data-driven functional are studied.

In the context of spline smoothing and model selection, Speckman
\cite{speckman1985spline} and Li \cite{li1986asymptotic, li1987asymptotic} have
obtained analogous results. Furthermore, Caponnetto and Yao's work on
semi-supervised statistical learning has yielded a consistency result
\cite{caponnetto2010cross}. Nevertheless, we will not employ the approach
proposed by Li, which is predicated on a comparison to Stein estimators.
Consequently, the result obtained will not be a straightforward generalization of
the approach from Li, but rather will take a different form. For instance, Li
demonstrated that generalized cross-validation is asymptotically optimal for
model selection, as the discretization dimension tends to infinity, while the
noise level $\delta$ , and the smoothness of the exact solution are kept fixed.
We show that generalized cross-validation is order-optimal for abstract mildly ill-posed
inverse problems, uniformly over $\delta$ and the smoothness of the true solution. Hereby a remarkable fact is that we do not have to assume self-similarity of the solution, see the next section for a more detailed dicussion of this point. While order-optimality is slightly weaker than asymptotic
optimality, it is generally the best that can be achieved for inverse problems.
Furthermore, the established bound is non-asymptotical. We then apply the result to an explicit ill-posed integral equation, analyzing in detail also the 
discretization error, a step that is frequently omitted.
The integral equation is formulated in an inherently infinite-dimensional
setting; however, through the finite number of measurement points, a semi-discrete
model is induced. We emphasize that the cross-validation method can only be formulated
in semi-discrete setting, and in most works, no error estimates of the
constructed estimator with respect to the continuous solution are provided.
In addition, we conduct  numerical experiments. First, we
recover a photograph that has been convolved with a Gaussian point-spread
function and subjected to noise addition. We demonstrate that the relative error
produced by the GCV-driven spectral cut-off converges monotonically as the noise
level decreases. A distinct plateau emerges solely for the minimal noise values,
thereby substantiating that the discretization error we previously analyzed
for one-dimensional integral equations, begins to dominate. In a subsequent
experiment, we implemented the GCV methodology in the context of parallel-beam
computed tomography. This approach yielded results in which the reconstruction
error and the selected truncation index remained within a few percent of the
optimal values.

\section{Setting and main result}

We consider the prototypical linear inverse problem
\begin{equation}\label{s1:e1}
  Kf = g.
\end{equation}
Here $K\in\mathcal L(X,Y)$ is a bounded, compact operator acting between real
Hilbert spaces $X$ and $Y$.
%
%
We have access to noisy observations
\begin{equation}\label{s1:e2}
g^\delta_{j,m}:=g^\dagger_{j,m}+\delta\varepsilon_j, \quad j=1,...,m,
\quad \mbox{and}\quad g_{m}^{\delta}:=(g_{j,m}^{\delta})_{j=1}^m,
\end{equation}
where $g^\dagger=Kf^\dagger$ is the unknown exact data,
$ g^\dagger_{j,m}=(P_mg^\dagger)_j, j=1,...,m$ are discretized not directly accessible
observations of $g^\dagger$ obtained by a bounded linear measurement operator
$P_m:\mathcal{R}(K)\to \R^m$, 
$\delta>0$ is the noise
level and
$\varepsilon_j$
are unbiased i.i.d random variables
with unit variance. The goal is to reconstruct the exact solution $f^\dagger$.
%
%
In the following, we assume that $K_m$ is injective for all $m\in\N$ and that $K$
is injective with dense range. Consequently, problem \eqref{s1:e1} is ill-posed
and therefore requires regularisation. To that end, we employ spectral methods
that utilize the spectral decomposition of the induced discretization of $K$. We
denote this discretization by $K_m$ and define it as follows:
%
%
\begin{align}\label{main:map}
  K_m: X &\;\rightarrow\;\R^m,\\
  f      &\;\mapsto\;\bigl(P_mKf\bigr)_{j=1}^m.
\end{align}
%
%
%
%
Denote by $(\sigma_{j,m},v_{j,m},u_{j,m})_{j=1}^m$ the singular value
decomposition of the compact operator $K_m$, i.e., $\sigma_{1,m}\ge ... \ge
\sigma_{m,m}>0$, $K_m v_{j,m}=\sigma_{j,m} u_{j,m}, K_m^*u_{j,m}=\sigma_{j,m}
v_{j,m}$ and $(v_{j,m})_{j=1}^m \subset \mathcal{N}(K_m)^\perp\subset X$ and
$(u_{j,m})_{j=1}^m \subset \R^m$ are orthonormal bases. The spectral cut-off
approximation to the unknown $f^\dagger$  is given by:
\begin{equation}
	f^\delta_{k,m}:= \sum_{j=1}^k \frac{\left(g^\delta_m,u_{j,m}\right)_{\R^m}}{\sigma_{j,m}} v_{j,m}
\end{equation}
and the ultimate goal will be to determine the truncation index $k\le m$ dependent only on $m$ (and without knowledge of $\delta$ or assumptions on the smoothness of $f^\dagger$). For the determination of the truncation index $k$ we choose generalized cross-validation due to Wahba. It is defined as follows:
\begin{align}
	k_m=k_{{\rm gcv},m}=k_m(\delta,f^\dagger,g^\delta_m)=\arg\min_{k=0,...,\frac{m}{2}}\frac{\sum_{j=k+1}^m(g^\delta_m,u_{j,m})_{\R^m}^2}{\left(1-\frac{k}{m}\right)^2}=:\arg\min_{k=0,...,\frac{m}{2}} \Psi_m(k).
\end{align}
This choice was introduced by Vogel \cite{vogel1986optimal} and can be derived
from the original leave-one-out method from Wahba \cite{wahba1977practical}, when
Tikhonov regularization is replaced with spectral cut-off regularization. The
only distinction from\cite{vogel1986optimal} is the restriction of the minimizing
set to $k\le m/2$ rather than $k\le m$. Other choices, such as $k\le
\frac{2}{3}m$, are also viable. In \cite{vogel1986optimal} such restriction
was unnecessary, as the expectation of the functional was considered.  I.e.,
there and in many other works the convergence analysis is not carried out for
$k_m$, but rather
$k_m^*=\arg\min_{k}\E[\Psi_m(k)]$. The results are usually that $k_m^*=(1+o(1))
\arg\min_{k}\E[S_m(k)]$ (as $m\to\infty$), where $S_m(k):=\E\|K_mf_{k,m}^\delta - K_mf^\dagger\|_{\R^m}^2]$ (that is, $k_m^*$ minimizes the weak (also called prediction) error),
under certain assumptions on the
singular value decomposition of $K, K_m$ and $f^\dagger$, and the constants
hidden in $o(1)$ are not given or unknown
\cite{Ju2021asymptotic, patil2024asymptotically}. In this
article we will investigate the strong error, i.e., with respect to $f^\dagger$,
under the truly data-driven choice $k_m$, and we will exactly
calculate all involved constants. In order to control the stochastic oscillations we set
\begin{equation*}
 p_{\varepsilon}(t):=\frac{3}{\varepsilon}\E\left[\left|\frac{1}{t}\sum_{j=1}^t(\varepsilon_j^2-1)\right|\right], \quad t\in\N, \quad \varepsilon\le \frac{1}{12}.
\end{equation*}
 Clearly, since the $\varepsilon_j$'s are unbiased with unit variance, we have
 $p_{\varepsilon}(t)\to0$ as $t\to\infty$. The rate of convergence can be
 controlled under additional assumptions, e.g., $p_\varepsilon(t) \le
 C/\sqrt{t}$, if $\E[\varepsilon_j^4] < \infty$.  In order to formulate our main
 result, we define the so-called weak and strong oracles for each
 individual $f^\dagger$:
\begin{align}\label{ow}
	t^\delta_m:&=t^\delta_m(f^\dagger):=\max\left\{0\le k\le m  ~:~ k \delta^2 \le \sum_{j=k+1}^m\sigma_{j,m}^2(f^\dagger,v_{j,m})^2\right\},\\\label{os}
	s^\delta_m:&= s^\delta_m(f^\dagger):=\max\left\{0 \le k \le m~:~\frac{k \delta^2}{\sigma_{k,m}^2} \le \sum_{j=k+1}^m(f^\dagger,v_{j,m})^2\right\}.
\end{align}
Since $\sigma_{j,m}\ge \sigma_{j+1,m}$ for all $j=1,..,m-1$, we immediately see
that $t_m^\delta \le s_m^\delta$. The terminology stems from the fact that
$t_m^\delta$ and $s_m^\delta$ balance the bias and variance in weak ($\E\|K_m
f^\delta_{k,m} - g_m^\delta\|_{\R^m}^2$) and strong norm ($\E\|
f^\delta_{k,m}-P_{\mathcal{N}(K_m)^\perp}f^\dagger\|^2$), respectively. Note
that, as we have access only to a finite-dimensional discretization, all we can
hope for is to approximate the projection of $f^\dagger$ onto the orthogonal
complement of the kernel of $K_m$.
Our main result is the following non-asymptotic error bound.
\begin{theorem}\label{l2}
 For $L_s:=\frac{\sqrt{1+\varepsilon}}{\varepsilon} + \sqrt{34\varepsilon+36}$ and uniformly for all $f^\dagger$ with $\frac{m}{2} \ge s_m^\delta(f^\dagger)\ge t_m^\delta(f^\dagger)\ge t\in\N$, it holds that
		\begin{align*}
&\mathbb{P}\left(\left\|f^\delta_{k_{\rm gcv},m} - P_{\mathcal{N}(K_m)^\perp}f^\dagger\right\| \le\frac{L_s\sqrt{s^\delta_m }\delta }{\sigma_{\frac{s^\delta_m}{\varepsilon^2},m}} \right)\\
&\qquad\ge 1-p_\varepsilon\left(\frac{2}{3}\frac{\varepsilon}{1+\varepsilon}t\right).
	\end{align*}
\end{theorem}
Note that the requirement $s_m^\delta(f^\dagger)\le \frac{m}{2}$ is made for
convenience and avoids considering a second summand in the numerator due to the
constraint $k^\delta_{\rm gcv}\le \frac{m}{2}$.

A key advantage of GCV is that it does not require any knowledge of the noise
level $\delta$. Therefore it belongs to the class of heuristic parameter choice
rules. The term heuristic stems from the fact that these rules provably do not
assemble convergent regularization schemes under a classical deterministic
worst-case noise model, due to the seminal work by Bakushinskii
\cite{bakushinskii1984remarks}. Still, for the white noise error model some
heuristic parameter choice rules, i.e., the quasi-opimality criterion and the
heuristic discrepancy principle yield convergent regularization methods, see
Bauer \& Rei\ss{} \cite{bauer2008regularization} and Jahn \cite{jahn2023noise}.
In order to prove mini-max optimality for those approaches, however additional
to the classical source condition the true solution must fulfill a
self-similarity condition, which is a substantial structural assumption as it
demands a concrete relation between the high and low frequency parts of the
unknown solution. Therefore it is remarkable that GCV yields mini-max
optimality without assuming self-similarity, as the following corollary shows. In
order to compare with the results from \cite{bauer2008regularization} and
\cite{jahn2023noise}, we assume here that the discretization scheme is along the
singular vectors of the continuous operator $K$, i.e., $ (P_m g)_j = (g,u_j)$. In
the next section we consider an exemplary integral equation, where, more
realistically, $P_m$ delivers point evaluations of the right hand side. For
mildly ill-posed problems, we derive the following oracle inequality, showing
that generalized cross-validation attains, up to a multiplicative constant, the
minimal error of all spectral-cut off estimators.
\begin{corollary}\label{cor1}
    Assume that $(P_m g)_j = (g,u_j)$ and $C_q j^{-q} \ge \sigma_j^2 \ge c_q j^{-q}$ for all $j\in\N$. Then, with $L_s$ from Theorem \ref{l2} and $L_s':=2\sqrt{L_s^2c_q^{-1}\varepsilon^{-2q}+2^{2q}+1} $ for all $f^\dagger$ with $\frac{m}{2} \ge s_m^\delta(f^\dagger)\ge t_m^\delta(f^\dagger) \ge t\in\N$, it holds that
    \begin{align*}
    &\mathbb{P}\left( \|f_{k_{\rm gcv},m}^\delta - f^\dagger\| \le L_s'  \min_{k=1,..,m}\sqrt{\E[\| f_{k,m}^\delta - f^\dagger\|^2]}\right)\\
    &~\ge 1 - p_\varepsilon\left(\frac{2}{3} \frac{\varepsilon}{1+\varepsilon} t\right).
    \end{align*}
\end{corollary}
Hereby, the restriction on the decay of the singular values is needed to guarantee that $\sigma_{\frac{s_m^\delta}{\varepsilon^2}}$ remains comparable to $\sigma_{s_m^\delta}$ in Theorem \ref{l2}. We guess that GCV is probably not consistent for general ill-posed
problems, as it might lack stability for exponentially falling singular values.
Such limitations regarding the robustness for exponentially ill-posed problems
have recently been studied for several related methods based on unbiased risk
estimation from Lucka \& al \cite{lucka2018risk}.

Achieving stability for exponentially ill-posed problems is possible with certain modifications of GCV. Those methods were
developed by Lukas and are called  robust and strong robust cross-validation,
see \cite{lukas2006robust} and \cite{lukas2008strong}. In both cases, the analysis is done in the weak sense, as explained above (that is, for an index minimizing the expectation of a certain functional).

\section{Application to an integral equation}
\label{applic}
We now analyse a concrete integral equation for which point evaluations of the
right-hand side are available. Consider
\begin{equation}\label{appl:int}
    (Kf)(x)=\int_0^1\kappa(x,y)f(y)\mathrm{d}x,
\end{equation}
where $K:L^2(0,1)\to L^2(0,1)$ is the compact linear operator with
continuous kernel $\kappa: (0,1)^2 \rightarrow \R$ given by $(x,y)\mapsto
\max(x(1-y),y(1-x))$. Note that continuity of
$\kappa$ implies that $Kf$ is continuous even if $f$ is only
square-integrable, hence point evaluations of $Kf$ are well defined. We assume a
uniform discretization with collocation points
$\{\xi_{j,m}\}_{j=1}^m  \subset (0,1)$, i.e.,
$\xi_{j,m}:=j/(m+1)$, and a set of noisy values
\begin{equation}\label{applic:colloc}
  g^\delta_{j,m}:=g^\dagger(\xi_{j,m})+\delta\varepsilon_j, \quad j=1,...,m,
  \quad \mbox{and}\quad g_{m}^{\delta}:=(g_{j,m}^{\delta})_{j=1}^m,
\end{equation}
of function $g^\dagger$ at these collocation points. Analogously to
\eqref{main:map}, we obtain a semi-discrete model $K_m$:
\begin{align*}
  K_m: L^2(0,1) &\;\rightarrow\;\R^m,\\
  f               &\;\mapsto\;\bigl((Kf)(\mathbf{\xi}_{j,m})\bigr)_{j=1}^m
  \;=\;\Bigl(\int_{0}^1\kappa(\mathbf{\xi}_{j,m},y)\,
  f(y)\,\mathrm{d}y\Bigr)_{j=1}^m.
\end{align*}

The setting here is particularly illustrative, since we can give the exact
singular value decomposition of $K$ and $K_m$:
\begin{proposition}\label{l1}
	For $\lambda_k:=\pi^2 k^2=:\sigma_k^{-1}$ and
	$v_k(x):=\sqrt{2}\sin(\sqrt{\lambda_k} x)$ there holds $K^*Kv_k=\sigma_k^2
	v_k$ for all $k\in\N$ and the $(v_k)_{k\in\N}$ form an orthonormal basis of
	$\mathcal{N}(K)^\perp=L^2(0,1)$. Moreover, for
	\begin{equation*}
    \sigma_{k,m}:=\frac{\sqrt{1-\frac{2}{3}\sin^2
    \left(\frac{\sqrt{\lambda_k}}{2(m+1)}\right)}}
    {4\sqrt{m+1}^3\sin^2\left(\frac{\sqrt{\lambda_k}}{2(m+1)}\right)}
	\end{equation*}
	and
	\begin{equation*}
v_{k,m}(\cdot):=\sum_{l=1}^m\sin\left(\sqrt{\lambda_k}\xi_l\right)\kappa(\xi_{l,m},
 \cdot)/\sigma_{k,m}\quad
 \mbox{and}\quad u_{k,m}:=\sqrt{\frac{2}{m+1}}\left( \sin(k\pi
\xi_{j,m})\right)_{j=1}^m
\end{equation*}
	 it holds that $K_mv_{k,m} = \sigma_{k,m} u_{k,m}$ and $K_m^* u_{k,m} =
\sigma_{k,m} v_{k,m}$, with $\left(v_{k,m}\right)_{k\le m}$ and
$\left(u_{k,m}\right)_{k\le m}$ orthonormal bases of
$\mathcal{N}(K_m)^\perp\subset L^2(0,1)$ and $\R^m$ respectively.
\end{proposition}
The proof will be given below in Section \ref{sec:lem}.

With this decomposition we can compute explicit error bounds, assuming that
$f^\dagger$ belongs to some unknown subset of $L^2$ with a certain smoothness.
For the kernel considered, we introduce the H\"older source classes
\begin{equation*}
\mathcal{X}_{s,\rho}:=\left\{f=(K^*K)^\frac{s}{2}h~:~h\in L^2,~\|h\|\le \rho
\right\}.
\end{equation*}
Proposition \ref{p4} below relates $\mathcal{X}_{s,\rho}$ to classical smoothness. We obtain the following error  bound.
\begin{theorem}\label{t0}
  Assume that $s> \frac{3}{4}$. Then, uniformly for all
  $f^\dagger\in\mathcal{X}_{s,\rho}$, the probability  that
  \begin{align*}
    &\| f^\delta_{k_{\rm gcv},m} - f^\dagger\|\\
    \le~& L_s \left(\frac{\delta}{\sqrt{m+1}}\right)^\frac{4s}{5+4s}
    \rho^\frac{5}{5+4s}+
    \frac{\|{f^\dagger}'\|}{\sqrt{2}(m+1)}\chi_{\{\frac{3}{4}<s\le
    \frac{5}{4}\}} + \frac{\|{f^\dagger}''\|}{2(m+1)^2}\chi_{\{s>\frac{5}{4}\}}
  \end{align*}
  is at least $1-p_{\varepsilon}\left(\frac{2}{3}
  \frac{\varepsilon}{\varepsilon+1}
  C_s\left(\frac{(m+1)\rho^2}{\delta^2}\right)^\frac{1}{5+4s}\right)$,
  where the constants $L_s$ and $C_s$ are given below in \eqref{t0c}
  and \eqref{t0c1}.
\end{theorem}
We briefly comment on the result. The first term in the upper bound resembles the
optimal convergence rate for the source condition $\mathcal{X}_{s,\rho}$ and is derived from Theorem \ref{l2}.  The remaining two
terms  bound the discretization error under different smoothness
$s$ of the exact solution and express how good the exact solution
$f^\dagger$ can be represented in the span of
$\kappa(\xi_{1,m},\cdot),...,\kappa(\xi_{m,m},\cdot)$ (note that those span the
space of piece-wise linear functions on the grid given by
$\xi_{1,m},...,\xi_{m,m}$). Note that the assumption $s>\frac{3}{4}$ imposes a differentiability condition onto the solution $f^\dagger$. If this
assumption is violated a similar bound will still hold, however it is not
possible to explicitly bound the aforementioned discretization error anymore.

\section{Proofs}
This section collects the proofs of the main results.
\subsection{Proof of Theorem \ref{l2}}
Before proving Theorem \ref{l2}, we establish an auxiliary proposition showing
that the probability of the following event, where the stochastic oscillations
behave 'nicely', tends to one:
\begin{align}\label{l2:e0}
    &\Omega_{t}:=\\\notag
    &~\left\{\omega\in\Omega~:~\left|\sum_{j=k+1}^l(g_m^\delta(\omega)-g_m^\dagger,u_{j,m})_{\R^m}^2-(l-k)\delta^2\right|\le \varepsilon (l-k)\delta^2,~\forall l\ge t,~k\le \frac{l}{2} \right\}.
\end{align}
\begin{proposition}\label{p0}
    The following bound holds for the probability $\Omega_{t}$ defined in
    \eqref{l2:e0}
    \begin{equation}\label{l2:e1}
        \mathbb{P}\left(\Omega_{t}\right) \ge 1- p_\varepsilon\left(\frac{2}{3} \frac{\varepsilon}{1+\varepsilon}t\right).
    \end{equation}
\end{proposition}

\begin{proof}[Proof of Proposition \ref{p0}]
  Define, for $\varepsilon':=\frac{2}{3}\frac{\varepsilon}{1+\varepsilon}$:
		\begin{equation*}
	 \Omega_{t}':=\left\{ \omega\in\Omega~:~\left|\frac{1}{l}\sum_{j=1}^l\left( (g_m^\delta(\omega)-g_m^\dagger,u_{j,m})_{\R^m}^2 - \delta^2\right)\right| \le \frac{\varepsilon}{3}\delta^2,~\forall l\ge \varepsilon' t \right\}.
	\end{equation*}
		Using the Kolmogorov--Doob inequality for backwards martingales one can prove
		that (see, e.g., Proposition 4.1 of \cite{jahn2022probabilistic})
		\begin{equation*}
	 \mathbb{P}\left(\Omega_t'\right) \ge 1- \frac{3}{\varepsilon}\E\left[ \left|\frac{1}{\varepsilon' t} \sum_{j=1}^{\varepsilon't} (\varepsilon_j^2 - 1)\right| \right] = 1-p_{\varepsilon}\left(\varepsilon' t\right).
	\end{equation*}
  It remains to show that $\Omega_t'\subset \Omega_t$. For this, we refine the
  argumentation in the proof of Proposition 3.1 of \cite{jahn2023noise}. For
  $l\ge t$ we distinguish two cases:
  \begin{enumerate}[label=\roman*).]
      \item If $l/2\ge k\ge \varepsilon' l$, then $k\ge
      \varepsilon't$ and thus
            \begin{align*}
                \sum_{j=k+1}^l\varepsilon_j^2\omkp& = \sum_{j=1}^l\varepsilon_j^2\omkp - \sum_{j=1}^k \varepsilon_j^2\omkp \le \left(1+\frac{\varepsilon}{3}\right)l - \left(1-\frac{\varepsilon}{3}\right)k \\
                & = (1+\varepsilon) (l-k) - \frac{2}{3}\varepsilon l + \frac{4}{3} \varepsilon k \le(1+\varepsilon)(l-k),
  	        \end{align*}
            since $k\le l/2$. Similarly, $\sum_{j=k+1}^l\varepsilon_j^2\omkp\ge
            (1-\varepsilon)(l-k)\omkp$.
        \item If $k<\varepsilon' l$, we obtain
            \begin{align*}
  	         \sum_{j=k+1}^l \varepsilon_j^2\omkp & \le \sum_{j=1}^l \varepsilon_j^2\omkp \le \left(1+\frac{\varepsilon}{3}\right) l  = (1+\varepsilon)(l-k) - \frac{2}{3}\varepsilon l + (1+\varepsilon)k\\
  	         &\le (1+\varepsilon)(l-k) - \frac{2}{3} \varepsilon l + (1+\varepsilon)\varepsilon'l = (1+\varepsilon)(l-k),
  	     \end{align*}
            by definition of $\varepsilon'$. Finally,
            \begin{align*}
                \sum_{j=k+1}^l\varepsilon_j^2\omkp&\ge \sum_{j=\varepsilon'l+1}^l\varepsilon_j^2\omkp \ge \left(1-\frac{\varepsilon}{3}\right) l\omkp - \left(1+\frac{\varepsilon}{3}\right)\varepsilon' l\omkp\\
                &= (1-\varepsilon)(l-k)\omkp + \left(\frac{2}{3}\varepsilon - \left(1+\frac{\varepsilon}{3}\right)\varepsilon'\right) l\omkp + (1-\varepsilon) k\omkp\\
                &= (1-\varepsilon)(l-k)\omkp + (1-\varepsilon)k \omkp\ge (1-\varepsilon)(l-k)\omkp.
  	        \end{align*}
  \end{enumerate}
    This proves $\Omega_t'\subset \Omega_t$ and therefore the claim \eqref{l2:e1}.
\end{proof}
\begin{remark}\label{r1}
If $l\ge t$, but $\frac{l}{2}<k\le l$, we will occasionally use the upper bound
\begin{equation*}
\sum_{j=k+1}^l(g_m^\delta-g_m^\dagger,u_{j,m})_{\R^m}^2\omk \le \sum_{j=1}^l(g_m^\delta-g_m^\dagger,u_{j,m})_{\R^m}^2\omk \le (1+\varepsilon)l\delta^2.
\end{equation*}
\end{remark}
We now prove the main results, fixing $\varepsilon\le\frac{1}{12}$.
   \begin{proof}[Proof of Theorem \ref{l2}]
		 Note that
	\begin{equation*}
	 (g_m^\dagger,u_{j,m})_{\R^m}=(K_mf^\dagger,u_{j,m})_{\R^m} = (f^\dagger,K_m^*u_{j,m}) = \sigma_{j,m}(f^\dagger,v_{j,m}).
\end{equation*}
	  We decompose the error as
		\begin{align}\notag
		f_{k^\delta_{{\rm gcv},m},m}^\delta - P_{\mathcal{N}^\perp(K_m)} f^\dagger &= \sum_{j=1}^{k^\delta_{{\rm gcv},m}} \frac{(g^\delta_m,u_{j,m})_{\R^m}}{\sigma_{j,m}} v_{j,m} - \sum_{j=1}^m (f^\dagger,v_{j,m}) v_{j,m}\\\label{err1}
		&= \sum_{j=1}^{k^\delta_{{\rm gcv},m}} \frac{(g_m^\delta-g_m^\dagger,u_{j,m})_{\R^m}}{\sigma_{j,m}} v_{j,m} - \sum_{j={k^\delta_{{\rm gcv},m}}+1}^m(f^\dagger,v_{j,m})v_{j,m}.
	\end{align}
In order to bound the first term, we claim an upper bound for the chosen index: For $t\le
t_m^\delta$ it holds that
\begin{equation}\label{claim1}
    k^\delta_{{\rm gcv},m}\omd \le \frac{t_m^\delta}{\varepsilon^2}.
    \end{equation}
    We prove the claim by showing that for all
    $\frac{t_m^\delta}{\varepsilon^2}<k\le \frac{m}{2}$ there holds
    \begin{equation}\label{p1:e1}
        \Psi_m(t_m^\delta)\omd< \Psi_m(k).
    \end{equation}
First, observe that
		\begin{align*}
		 &\Psi_m(t_m^\delta)\omd =\frac{\sum_{j=t_m^\delta+1}^m(g^\delta_m,u_{j,m})_{\R^m}^2}{\left(1-\frac{t^\delta_m}{m}\right)^2}\omd \\
     &~= \frac{\sum_{j=t^\delta_m+1}^k(g_m^\delta,u_{j,m})_{\R^m}^2}{\left(1-\frac{t^\delta_m}{m}\right)^2}\omd + \frac{\sum_{j=k+1}^m(g_m^\delta,u_{j,m})_{\R^m}^2}{\left(1-\frac{t^\delta_m}{m}\right)^2}\omd\\
		 &~\le \frac{\left(\sqrt{\sum_{j=t^\delta_m+1}^k(g_m^\delta - g_m^\dagger,u_{j,m})_{\R^m}^2} + \sqrt{\sum_{j=t^\delta_m+1}^k(g_m^\dagger,u_{j,m})_{\R^m}^2}\right)^2} {\left(1-\frac{t^\delta_m}{m}\right)^2}\omd  \\
     &\qquad + \left(\frac{1-\frac{k}{m}}{1-\frac{t^\delta_m}{m}}\right)^2 \Psi_m(k)\\
	    &~\le \frac{\left((1+\varepsilon)\sqrt{k}\delta + \sqrt{t^\delta_m}\delta\right)^2}{\left(1-\frac{t^\delta_m}{m}\right)^2} + \left(\frac{m-k}{m-t^\delta_m}\right)^2 \Psi_m(k)\\
	   &~\le \frac{\left((1+\varepsilon)\sqrt{k}\delta + \sqrt{\varepsilon^2k}\delta\right)^2}{\left(1-\frac{t^\delta_m}{m}\right)^2} + \left(\frac{m-k}{m-t^\delta_m}\right)^2 \Psi_m(k)\\
	   &~\le \frac{(1+2\varepsilon)^2k\delta^2}{\left(1-\frac{t^\delta_m}{m}\right)^2} + \left(\frac{m-k}{m-t^\delta_m}\right)^2 \Psi_m(k)
		\end{align*}
		Observe that $k\le m-k$ and $t_m^\delta\le \varepsilon^2 k$. Then, on the
		other hand,
		\begin{align*}
			\Psi_m(k)\omd &\ge \frac{\left(\sqrt{\sum_{j=k+1}^m(g_m^\delta - g_m^\dagger,u_{j,m})_{\R^m}^2}-\sqrt{\sum_{j=k+1}^m(g_m^\dagger,u_{j,m})_{\R^m}^2}\right)^2}{\left(1-\frac{k}{m}\right)^2}\omd\\
			&\ge \frac{\left((1-\varepsilon)\sqrt{m-k} \delta-\sqrt{t_m^\delta}\delta\right)^2}{\left(1-\frac{k}{m}\right)^2}\omd\ge \frac{\left((1-\varepsilon)\sqrt{m-k}\delta-\varepsilon\sqrt{k}\delta\right)^2}{\left(1-\frac{k}{m}\right)^2}\omd\\
			&\ge \frac{\left((1-\varepsilon)\sqrt{m-k}\delta-\varepsilon\sqrt{m-k}\delta\right)^2}{\left(1-\frac{k}{m}\right)^2}\omd\ge \frac{(1-2\varepsilon)^2 (m-k)\delta^2}{\left(1-\frac{k}{m}\right)^2}\omd
		\end{align*}
				We solve the second inequality for $\delta$ and plug into the first equation and obtain
				\begin{align*}
			\Psi_m(t_m^\delta)\omd &\le \frac{(1+2\varepsilon)^2k\delta^2}{\left(1-\frac{t^\delta_m}{m}\right)^2}\omd + \left(\frac{m-k}{m-t^\delta_m}\right)^2 \Psi_m(k)\\
			&\le \frac{(1+2\varepsilon)^2k}{\left(1-\frac{t_m^\delta}{m}\right)^2} \frac{\left(1-\frac{k}{m}\right)^2}{(1-2\varepsilon)^2(m-k)}\Psi_m(k) + \left(\frac{m-k}{m-t_m^\delta}\right)^2\Psi_m(k)\\
			&= \Psi_m(k)\frac{m-k}{(m-t_m^\delta)^2}\left(k \left(\frac{1+2\varepsilon}{1-2\varepsilon}\right)^2+m-k\right)\\
			&=\Psi_m(k)\frac{m^2 -\left(2-\left(\frac{1+2\varepsilon}{1-2\varepsilon}\right)^2\right) mk -\left( \left(\frac{1+2\varepsilon}{1-2\varepsilon}\right)^2-1\right) k^2}{m^2 - 2m t_m^\delta+ {t_m^\delta}^2}\\		&< \Psi_m(k)\frac{m^2 -\left(2-\left(\frac{1+2\varepsilon}{1-2\varepsilon}\right)^2\right) mk}{m^2 - 2m t_m^\delta}<\Psi_m(k),
		\end{align*}
				since
					\begin{align*}
			\frac{2-\left(\frac{1+2\varepsilon}{1-2\varepsilon}\right)^2}{2}\frac{k}{t_m^\delta}\ge\frac{2-\left(\frac{1+2\varepsilon}{1-2\varepsilon}\right)^2}{2\varepsilon^2}> 1
		\end{align*}
			for $\varepsilon\le 1/12$. This proves \eqref{p1:e1} and thus claim \eqref{claim1}.

            Therefore, applying Remark \ref{r1} and the claim \eqref{claim1}, we can bound the norm of the first term in \eqref{err1} as follows:
    \begin{align}\notag
		\sum_{j=1}^{k^\delta_{{\rm gcv},m}} \frac{\left(g_m^\delta-g_m^\dagger,u_{j,m}\right)^2_{\R^m}}{\sigma_{j,m}^2}\omd \quad &\le \quad \frac{1}{\sigma_{k_{\rm gcv},m}^2} \sum_{j=1}^{k^\delta_{{\rm gcv},m}} (g_m^\delta-g_m^\dagger,u_{j,m})^2_{\R^m}\omd \\
        \le \quad (1+\varepsilon) \frac{k^\delta_{{\rm gcv},m}\delta^2}{\sigma_{k_{{\rm gcv},m}}^2}\omd\label{err2a}
		\quad &\le \quad \frac{1+\varepsilon}{\varepsilon^2} \frac{t_m^\delta\delta^2}{\sigma_{\frac{t_m^\delta}{\varepsilon^2}}^2}.
	\end{align}
In order to bound the second term in \eqref{err1} we first claim that
		\begin{equation}\label{claim2a}
		\sum_{j=k^\delta_{{\rm gcv},m}+1}^m(g_m^\dagger,u_{j,m})_{\R^m}^2 \omd \le C_a t_m^\delta \delta^2
		\end{equation}
				with $C_a:=35+34 \varepsilon$. We prove the claim. 		Since $C_a \ge 1$ the
				assertion clearly holds for $k_{{\rm gcv},m}^\delta \omd > t_m^\delta$ by
				definition of $t_m^\delta$. So  assume $k^\delta_{{\rm gcv},_m}<
				t_m^\delta$ and recall that $t_m^\delta\le m/2$. Then, by definition of
				$t_m^\delta$,
            \begin{align*}
			     ~&\sum_{j=k^\delta_{{\rm
			     gcv},_m}+1}^{m}(g_m^\dagger,u_{j,m})_{\R^m}^2\omd
			     = \sum_{j=k^\delta_{{\rm gcv},_m}+1}^{t_m^\delta}
			     (g_m^\dagger,u_{j,m})_{\R^m}^2\omd
			     +\sum_{j=t_m^\delta+1}^m(g_m^\dagger,u_{j,m})_{\R^m}^2\\
                 \le ~& \sum_{j=k_{{\rm gcv},m}^\delta+1}^{t_m^\delta}
                 (g_m^\dagger,u_{j,m})_{\R^m}^2\omd + t_m^\delta \delta^2\\
			     \le~& 2 \sum_{j=k_{{\rm gcv},m}+1}^{t_m^\delta}(g^\delta_m,u_{j,m})_{\R^m}^2 + 2 \sum_{j=k^\delta_{{\rm gcv},m}+1}^{t_m^\delta}(g_m^\delta-g_m^\dagger,u_{j,m})_{\R^m}^2\omd + t_m^\delta\delta^2\\
			     \le~& 2\sum_{j=k^\delta_{{\rm gcv},m}+1}^{t_m^\delta}(g^\delta_m,u_{j,m})_{\R^m}^2 + (3+2\varepsilon) t_m^\delta \delta^2.
		      \end{align*}
			Because
			\begin{align*}
			     \Psi_m(k^\delta_{{\rm gcv},m}) &= \frac{\sum_{j=k^\delta_{{\rm gcv},m}+1}^m(g^\delta_m,u_{j,m})_{\R^m}^2}{\left(1-\frac{k^\delta_{{\rm gcv},m}}{m}\right)^2} \\
                &= \frac{\sum_{j=k^\delta_{{\rm gcv},m}+1}^{t_m^\delta}(g^\delta_m,u_{j,m})_{\R^m}^2}{\left(1-\frac{k^\delta_{{\rm gcv},m}}{m}\right)^2}  + \frac{\sum_{j=t^\delta_m+1}^{m}(g^\delta_m,u_{j,m})_{\R^m}^2}{\left(1-\frac{k^\delta_{{\rm gcv},m}}{m}\right)^2} \\
			     &= \frac{\sum_{j=k^\delta_{{\rm gcv},m}+1}^{t_m^\delta}(g^\delta_m,u_{j,m})_{\R^m}^2}{\left(1-\frac{k^\delta_{{\rm gcv},m}}{m}\right)^2}  + \left(\frac{m-t_m^\delta}{m-k^\delta_{{\rm gcv},m}}\right)^2\Psi_m(t_m^\delta)
		      \end{align*}
			we conclude, since $k^\delta_{{\rm gcv},m}$ is the minimizer of $\Psi_m$ on $0\le k\le m/2$ and $t_m^\delta\le \frac{m}{2}$,
			\begin{align*}
			     &\sum_{j=k^\delta_{{\rm gcv},m}+1}^{t_m^\delta}(g^\delta_m,u_{j,m})_{\R^m}^2\omd\\
			     =~&\left(1-\frac{k^\delta_{{\rm gcv},m}}{m}\right)^2\Psi_m(k^\delta_{{\rm gcv},m})\omd - \left(1-\frac{t_m^\delta}{m}\right)^2\Psi_m(t_m^\delta)\omd\\
			     \le~&\left(1-\frac{k^\delta_{{\rm gcv},m}}{m}\right)^2\Psi_m(t_m^\delta)\omd - \left(1-\frac{t_m^\delta}{m}\right)^2\Psi_m(t_m^\delta)\omd\\
			     =~&\frac{\Psi_m(t_m^\delta)}{m}\left(2t_m^\delta-2k^\delta_{{\rm gcv},m} +\frac{{k^\delta_{{\rm gcv},m}}^2-{t_m^\delta}^2}{m}\right)\omd\le \frac{2t_m^\delta\Psi_m(t_m^\delta)}{m}\omd\\
			     =~&\frac{2 t_m^\delta}{m} \frac{\sum_{j=t_m^\delta+1}^m(g^\delta_m,u_{j,m})_{\R^m}^2}{\left(1-\frac{t_m^\delta}{m}\right)^2}\\
			     \le~& \frac{4t_m^\delta}{m} \frac{\sum_{j=t_m^\delta+1}^m(g_m^\delta-g_m^\dagger,u_{j,m})_{\R^m}^2\omd + \sum_{j=t_m^\delta+1}^m(g_m^\dagger,u_{j,m})_{\R^m}^2}{\left(1-\frac{t_m^\delta}{m}\right)^2}\\
			     \le~& \frac{4t_m^\delta}{m}\frac{(1+\varepsilon)(m-t_m^\delta)\delta^2 + t_m^\delta\delta^2}{\left(1-\frac{t_m^\delta}{m}\right)^2} \le 4(1+\varepsilon) \frac{t_m^\delta \delta^2}{\frac{1}{2^2}} = 16(1+\varepsilon) t_m^\delta \delta^2.
		      \end{align*}
			Putting everything together we obtain
			\begin{align*}
			     \sum_{j=k^\delta_{{\rm gcv},m}+1}^{m}(g_m^\dagger,u_{j,m})_{\R^m}^2\omd &\le 32(1+\varepsilon)t_m^\delta\delta^2 + (3+2\varepsilon)t_m^\delta\delta^2 &&\\
                 &= (35+34 \varepsilon) t_m^\delta \delta^2 = C_a t_m^\delta \delta^2 
		      \end{align*}
              and thus proved claim \eqref{claim2a}.

We are now in a position to bound the norm of the second term in \eqref{err1}. We
have
		\begin{align}\notag
		\sum_{j=k^\delta_{{\rm gcv},m}+1}^m(f^\dagger,v_{j,m})\omd&=\sum_{j=k^\delta_{{\rm gcv},m}+1}^{s_m^\delta}(f^\dagger,v_{j,m})^2\omd  + \sum_{j=s_m^\delta+1}^m(f^\dagger,v_{j,m})^2\\\notag
		&\le \frac{1}{\sigma_{s^\delta_m,m}^2}\sum_{k^\delta_{{\rm gcv},m}+1}^{s_m^\delta} (g^\dagger,u_{j,m})^2_{\R^m}\omd	 +\frac{s_m^\delta \delta^2}{\sigma_{s_m^\delta,m}^2}\\\notag
		&\le \frac{1}{\sigma_{s^\delta_m,m}^2}\sum_{k^\delta_{{\rm gcv},m}+1}^{m} (g^\dagger,u_{j,m})^2_{\R^m}\omd	 +\frac{s_m^\delta \delta^2}{\sigma_{s_m^\delta,m}^2}\\\label{err2b}
        &\le \frac{C_a t_m^\delta + s_m^\delta}{\sigma^2_{s_m^\delta,m}}\delta^2.
	\end{align}
		Combining \eqref{err2a} and \eqref{err2b} with $t_m^\delta(f^\dagger)\le
		s_m^\delta(f^\dagger)$, we can bound the norm of the error \eqref{err1} and
		conclude
		\begin{align*}
		\left\|f_{k^\delta_{{\rm gcv},m},m}^\delta - P_{\mathcal{N}^\perp(K_m)}
		f^\dagger\right\|\omd &\le  \frac{L_s\sqrt{s_m^\delta}
		\delta}{\sigma_{\frac{s_m^\delta}{\varepsilon^2},m}},
	\end{align*}
		where $L_s:=\frac{\sqrt{1+\varepsilon}}{\varepsilon} + \sqrt{C_a+1}$. This
		completes the proof of Theorem \ref{l2}.
	\end{proof}

\subsection{Proof of Corollary \ref{cor1}}
For this discretization we have $\sigma_{j,m}=\sigma_j$ and $v_{j,m}=v_j$,
$u_{j,m}=u_j$ for $j=1,...,m$. Consequently, 
\begin{align*}
\|f^\delta_{k_{\rm gcv},m}
-f^\dagger\|^2 &= \|f^\delta_{k_{\rm gcv},m} - P_{\mathcal{N}(K_m)^\perp}
f^\dagger\|^2 + \sum_{j=m+1}^\infty (f^\dagger,v_j)^2 \\
&\le \|f_{k_{\rm gcv},m}-
P_{\mathcal{N}(K_m)^\perp}f^\dagger\|^2 + \frac{s_m^\delta
\delta^2}{\sigma_{s_m^\delta}^2}
\end{align*}
and with Theorem \ref{l2} we deduce
\begin{align*}
    \mathbb{P}\left(\|f^\delta_{k_{\rm gcv},m} - f^\dagger\| \le \sqrt{L_s^2c_q^{-1}\varepsilon^{-2q}+2^{2q}+1} \frac{\sqrt{s_m^\delta} \delta}{\sigma_{s_m^\delta}}\right) \ge 1- p_\varepsilon\left(\frac{2}{3}\frac{\varepsilon}{1+\varepsilon} t\right).
\end{align*}
The proof follows by observing that
\begin{align*}
&\E\|f^\delta_{k,m} - f^\dagger\|^2 = \delta^2\sum_{j=1}^k \sigma_j^{-2} + \sum_{j=k+1}^\infty (f^\dagger,v_j)^2 \ge \frac{\delta^2 (k+1)^{q+1}}{(q+1)C_q} + \sum_{j=k+1}^\infty (f^\dagger,v_j)^2 \\
&~= \frac{2^q}{q+1} \frac{k\delta^2}{\sigma_k^2}+\sum_{j=k+1}^\infty (f^\dagger,v_j)^2 \ge \frac{1}{2}\frac{s_m^\delta \delta^2}{\sigma_{s_m^\delta}^2},
\end{align*}
where the last inequality follows from the fact that $(a_k+b_k)_{k\in\N}$, where $a_k$ is (polynomially) monotonically increasing and $b_k$ is monotonically decreasing, is lower bounded by $\frac{1}{2}(a_{k^*}+b_{k^*})$, with the balancing index $k^*:= \min \{k\ge 0~:~ a_k \ge b_k\}$.

\subsection{Proof of Theorem \ref{t0}}
The proof requires bounds for the discretization
error $\|P_{\mathcal{N}(K_m)^\perp} f^\dagger - f^\dagger\|$ and fo the balancing oracle $s_m^\delta$.

We start with formulating the following technical claim, which will be proved in
the appendix: For $j=t(m+1)+s$ with $m\in\N, t\in\N_0$ and $s\in\{0,...,m\},
k\in\{1,...,m\}$, we have
    \begin{equation}\label{claim:sv}
    (v_j,v_{k,m}) = \sqrt{m+1}\frac{\sigma_j}{\sigma_{k,m}} \begin{cases} 1 &\quad \mbox{for }s=k~\mbox{and}~t~\mbox{even,}  \\
        -1 &\quad \mbox{for }s+k=m+1~\mbox{and}~t~\mbox{odd,}\\
        0 &\quad \mbox{else.}\end{cases}
        \end{equation}
We use this claim to derive the bound on the balancing oracles.	It holds that
		\begin{align}\label{claim:oracle}
		\sup_{f\in \mathcal{X}_{\nu,\rho}} s_m^\delta(f) &\le C_s  \left(\frac{(m+1)\rho^2}{\delta^2}\right)^\frac{1}{5+8s},\\\label{claim:oracle1}
		\sup_{f\in \mathcal{X}_{\nu,\rho}} t_m^\delta(f) &\le C_s  \left(\frac{(m+1)\rho^2}{\delta^2}\right)^\frac{1}{5+8s}.
	\end{align}
		with $C_s$ given below.
		By \eqref{claim:sv}, it holds that
		\begin{align*}
		v_{j,m} &= \sum_{l=1}^\infty (v_{j,m},v_l) v_l\\
		 &=\sqrt{m+1}\frac{\sigma_j}{\sigma_{j,m}} v_j - \sqrt{m+1}\sum_{t=1}^\infty   \frac{\sigma_{2t(m+1)-j} v_{2t(m+1)-j}-\sigma_{2t(m+1)+j} v_{2t(m+1)+j}}{\sigma_{j,m}}.
	\end{align*}
		Therefore, with $f^\dagger=\sum_{j=1}^\infty \varphi(\sigma_j^2)(h,v_j) v_j =:\sum_{j=1}^\infty f_j v_j$, we obtain
		\begin{align*}
		&(f^\dagger,v_{j,m}) = \sum_{l=1}^\infty f_l (v_l,v_{j,m}) \\
        &= \frac{\sqrt{m+1}}{\sigma_{j,m}}\Bigg( \sigma_j f_j  - \sum_{t=1}^\infty\left(\sigma_{2t(m+1)-j}f_{2t(m+1)-j} - \sigma_{2t(m+1)+j}f_{2t(m+1)+j}\right)\Bigg)\\
		&=\varphi_s(\sigma_{j,m}^2) \sqrt{m+1} \frac{\sigma_j\varphi_s(\sigma_j^2)}{\sigma_{j,m}\varphi_s(\sigma_{j,m}^2)} *\\
        &\quad *\Bigg(\vphantom{\sum_{t=1}^\infty\frac{1}{2}} (h,v_j)  - \sum_{t=1}^\infty\frac{\sigma_{2t(m+1)-j}\varphi_s(\sigma_{2t(m+1)-j})(h,v_{2t(m+1)-j})} {\sigma_j \varphi_s(\sigma_j^2)} \\
        &\quad \qquad \qquad + \sum_{t=1}^\infty \frac{ \sigma_{2t(m+1)+j}\varphi_s(\sigma_{2t(m+1)+j})(h,v_{2t(m+1)+j})}{\sigma_j \varphi_s(\sigma_j^2)} \Bigg)
        =:\varphi_s(\sigma_{j,m}^2)w_{j,m,s}.
	\end{align*}
		Using the Cauchy-Schwarz-inequality gives
		\begin{align*}
		\Bigg(w_{j,m,s}&\frac{\sigma_{j,m}\varphi_s(\sigma_{j,m}^2)}{\sqrt{m+1}\sigma_j\varphi_s(\sigma_j^2)}\Bigg)^2 \le &&\\
        &2(h,v_j)^2
		  + 2 \left(\sum_{t=1}^\infty  \Bigg(\frac{\sigma_{2t(m+1)-j}^2}{\sigma_j^2}\right)^\frac{s+1}{2}|(h,v_{2t(m+1)-j})| &\\
         & \qquad + \left(\frac{\sigma_{2t(m+1)+j}^2}{\sigma_j^2}\right)^\frac{s+1}{2}|(h,v_{2t(m+1)+j})|\Bigg)^2.&
		 \end{align*}
	 	 For the second term, using again Cauchy-Schwarz and the fact that $1/(x-1)\le 2/x$ for any $x\ge 2$, we further obtain
	 		 \begin{align*}
	 		 	&\left(\sum_{t=1}^\infty \left(\frac{\sigma_{2t(m+1)-j}^2}{\sigma_j^2}\right)^\frac{s+1}{2}|(h,v_{2t(m+1)-j})| + \left(\frac{\sigma_{2t(m+1)+j}^2}{\sigma_j^2}\right)^\frac{s+1}{2}|(h,v_{2t(m+1)+j})|\right)^2\\
	 	&~=  \left(\sum_{t=1}^\infty \left(\frac{1}{2t\frac{m+1}{j} -1}\right)^{2s+2}|(h,v_{2t(m+1)-j})| + \left(\frac{1}{2t\frac{m+1}{j}+1}\right)^{2s+2}|(h,v_{2t(m+1)+j})|\right)^2\\
	 	&~\le  \left(\sum_{t=1}^\infty \left(\frac{1}{2t\frac{m+1}{j} -1}\right)^{4s+4} + \left(\frac{1}{2t\frac{m+1}{j}+1}\right)^{4s+4}\right) \times \\
        &~\qquad \times \left( \sum_{t=1}^\infty (h,v_{2t(m+1)-j})^2 +(h,v_{2t(m+1)+j})^2\right)\\
            &~\le \left(\sum_{t=1}^\infty \left(\frac{2}{2t\frac{m+1}{j} }\right)^{4s+4} + \left(\frac{1}{2t\frac{m+1}{j}}\right)^{4s+4}\right)\left( \sum_{t=1}^\infty (h,v_{2t(m+1)-j})^2 +(h,v_{2t(m+1)+j})^2\right)\\
	 	&~\le  2 \left( \sum_{t=1}^\infty t^{-4s-4} \right)\left( \sum_{t=1}^\infty (h,v_{2t(m+1)-j})^2 +(h,v_{2t(m+1)+j})^2\right)\\
	 	&~\le \frac{2}{4s+3} \left( \sum_{t=1}^\infty (h,v_{2t(m+1)-j})^2 +(h,v_{2t(m+1)+j})^2\right)
	 \end{align*}
  and finally
 		 \begin{align*}
	 	\Bigg((h,v_j)\vphantom{\sum_{t=1}^\infty\frac{1}{2}}
		 	&- \sum_{t=1}^\infty\frac{\sigma_{2t(m+1)-j}\varphi_s(\sigma_{2t(m+1)-j})(h,v_{2t(m+1)-j})}{\sigma_j \varphi_s(\sigma_j^2)}\\
            &+\sum_{t=1}^\infty \frac{\sigma_{2t(m+1)+j}\varphi_s(\sigma_{2t(m+1)+j})(h,v_{2t(m+1)+j})}{\sigma_j \varphi_s(\sigma_j^2)}\Bigg)^2\\
		\le &2\left((h,v_j)^2 + \sum_{t=1}^\infty (h,v_{2t(m+1)-j})^2 +(h,v_{2t(m+1)+j})^2\right).
	\end{align*}
		Moreover, we use $\sin^2(x)\in[0,1]$ and $\sin(x)\le x$  and obtain
		\begin{align*}
		(m+1) \frac{\sigma_j^2\varphi_s^2(\sigma_j^2)}{\sigma_{j,m}^2\varphi_s^2(\sigma_{j,m}^2)}&= (m+1)\left(\frac{\sigma_j^2}{\sigma_{j,m}^2}\right)^{s+1} \\
        &= (m+1)\left(\frac{16(m+1)^3\sin^4\left(\frac{j\pi}{2(m+1)}\right)}{\pi^4 j^4 \left(1-\frac{2}{3}\sin^2\left(\frac{j\pi}{2(m+1)}\right)\right)}\right)^{s+1}\\
		&\le (m+1)\left(\frac{3}{(m+1)}\right)^{s+1}= \frac{3^{s+1}}{(m+1)^{s}},
	\end{align*}
    thus
    \begin{equation}\label{est:coef}
    w_{j,m,s}^2 \le    \frac{2*3^{s+1}}{(m+1)^s} \left((h,v_j)^2 +  \sum_{t=1}^\infty (h,v_{2t(m+1)-j})^2 +(h,v_{2t(m+1)+j})^2\right)
    \end{equation}
		Putting all estimates together yields
	\begin{align}\notag
		&\sum_{j=k+1}^m(f,v_{j,m})^2\\\notag
		 \le~& 2* 3^{s+1} \frac{\varphi_s^2\left(\sigma_{k+1,m}^2\right)}{(m+1)^{s}} \sum_{j=k+1}^m\left((h,v_j)^2 +  \sum_{t=1}^\infty (h,v_{2t(m+1)-j})^2 +(h,v_{2t(m+1)+j})^2\right)\\\notag
		\le~& \frac{2* 3^{s+1}}{(m+1)^{s}} \varphi_s^2\left(\frac{1-\frac{2}{3}\sin^2\left(\frac{(k+1)\pi}{2(m+1)}\right)}{16(m+1)^3\sin^4\left(\frac{(k+1)\pi}{2(m+1)}\right)}\right) \sum_{l=k+1}^\infty(h,v_l)^2\\\label{e000}
		\le~& \frac{2* 3^{s+1}}{(m+1)^{s}} \varphi_s^2\left(\frac{m+1}{2^4(k+1)^4}\right) \rho^2 = \frac{3^{s+1}}{2^{4s-1}} k^{-4s} \rho^2,
	\end{align}
		where we used that $\sin(x)\ge \frac{2}{\pi}x$ for $0\le x\le \frac{\pi}{2}$ in the third step and the fact that for every $l\ge m+1$ there is at most one pair $(j,t)$ such that $l=2t(m+1)-j$ or $l=2t(m+1)+j$ in the second step. Therefore, on the one hand,
		\begin{align*}
		\sup_{f^\dagger\in\mathcal{X}_{s,\rho}} \sum_{j=k+1}^\infty (f^\dagger,v_{j,m})^2\le \frac{3^{s+1}}{2^{4s-1}}k^{-4s}\rho^2,
	\end{align*}
		while on the other hand
		\begin{align*}
		\frac{k \delta^2}{\sigma_{k,m}^2}& = \frac{16 (m+1)^3\sin^4\left(\frac{k\pi}{2(m+1)}\right)}{1-\frac{2}{3}\sin^2\left(\frac{k\pi}{2(m+1)}\right)} k \delta^2\le \frac{16 \pi^4 k^4}{\frac{1}{3} 2^4(m+1)} k \delta^2 = 3 \pi^4 \frac{k^5\delta^2}{m+1}.
	\end{align*}
		Consequently,
		\begin{align*}
		3\pi^4 \frac{k \delta^2}{m+1} &\stackrel{!}{\le} \frac{3^{s+1}}{2^{4s-1}} k^{-4s} \rho^2\\
		\Longrightarrow\quad k&\le C_s \left(\frac{(m+1)\rho^2}{\delta^2}\right)^\frac{1}{5+4s}
	\end{align*}
		with
		\begin{equation}\label{t0c1}
	C_s:=\left(\frac{3^{s}}{2^{4s-1}\pi^4}\right)^\frac{1}{5+4s}.
	\end{equation}
	 We conclude
		\begin{equation*}
		\sup_{f^\dagger\in\mathcal{X}_{s,\rho}} s_m^\delta(f^\dagger)\le C_s  \left(\frac{(m+1)\rho^2}{\delta^2}\right)^\frac{1}{5+4s}.
	\end{equation*}
	With similar arguments we also get
	\begin{equation*}
		\sup_{f^\dagger\in\mathcal{X}_{s,\rho}} t_m^\delta(f^\dagger)\le C_s  \left(\frac{(m+1)\rho^2}{\delta^2}\right)^\frac{1}{5+4s},
	\end{equation*}
and the proof of claims \eqref{claim:oracle} and \eqref{claim:oracle1} is finished.

 	Thus, for $\delta$ not too large we have $t_m^\delta(f^\dagger)\le m/2$. Applying Theorem \ref{l2}, for $t_m^\delta\ge t$ we therefore obtain, with \eqref{e000},
		\begin{align}\notag
		&\| f^\delta_{k^\delta_{{\rm gcv},m},m} - P_{\mathcal{N}(K_m)^\perp}  f^\dagger\|\omd\\
		 \le~&	\frac{L_s\sqrt{s^\delta_m }\delta }{\sigma_{\frac{s^\delta_m}{\varepsilon^2},m}}\le \frac{\sqrt{3}L_s \pi^2}{\varepsilon^4} {s_m^\delta}^\frac{5}{2}\frac{\delta}{\sqrt{m+1}}\le\frac{\sqrt{3}C_s^\frac{5}{2} L_s \pi^2}{\varepsilon^4} \left(\frac{(m+1)\rho^2}{\delta^2}\right)^\frac{5}{2(5+4s)}\frac{\delta}{\sqrt{m+1}}\\\label{t0:e1}
		=~ &L_s' \left(\frac{\delta}{\sqrt{m+1}}\right)^\frac{4s}{5+8s} \rho^\frac{5}{5+4s},
	\end{align}
		with
		\begin{equation}\label{t0c}
	L_s:= \frac{\sqrt{3}C_s^\frac{5}{2} L_s\pi^2}{\varepsilon^4}.
	\end{equation}
		Finally, we treat the discretization error $\|P_{\mathcal{N}^\perp(K_m)} f^\dagger- f^\dagger\|$. First, by definition of $\kappa$ we see that the span $<v_{1,m},...,v_{m,m}>$ is equal to the space of piece-wise linear functions on the grid $\xi_{1,m},...,\xi_{m,m}$,  and  $f_m^\dagger=P_{\mathcal{N}(K_m)^\perp}f^\dagger$ is the $L^2$-projection of $f^\dagger$ onto that space.  The error depends on classical smoothness of $f^\dagger$ which can be related to the H\"older source condition in Proposition \ref{p4} below.
Then, classical estimates for the linear interpolating spline yield the following bound for the discretization error,
		\begin{align}\label{t0e2}
		\|P_{\mathcal{N}(K_m)^\perp} f^\dagger-f^\dagger\|_{L^2}&\le \begin{cases} \frac{\|(f^\dagger)'\|_{L^2}}{\sqrt{2}(m+1)},\quad &\mbox{for } s\ge\frac{3}{4}\\
			\frac{\|(f^\dagger)''\|_{L^2}}{2(m+1)^2},\quad &\mbox{for } s\ge \frac{5}{4}
		\end{cases}.
	\end{align}
	Finally, plugging the estimates \eqref{t0:e1} and \eqref{t0e2} into the decomposition
\begin{align*}
\|f^\delta_{k^\delta_{\rm gcv},m} - f^\dagger\|\omk &\le \|f^\delta_{k^\delta_{\rm gcv},m} - P_{\mathcal{N}(K_m)^\perp} f^\dagger\|\omk +  \|P_{\mathcal{N}(K_m)^\perp} f_m^\dagger- f^\dagger\|\omk
\end{align*}
		and applying Proposition \ref{p0} finishes the proof of
		Theorem \ref{t0}.

        			\begin{proposition}\label{p4}
		Assume that $f^\dagger \in\mathcal{X}_{s,\rho}$. If $s>\frac{3}{4}$, then $f^\dagger$ is differentiable. if $s>\frac{5}{4}$, then $f^\dagger$ is twice differentiable.
	\end{proposition}
	\begin{proof}[Proof of Proposition \ref{p4}]
		First $f^\dagger\in\mathcal{X}_{s,\rho}$ implies that there exists $h\in L^2$ with $\|h\|\le \rho$, such that $f^\dagger=\sum_{j=1}^\infty \varphi_s(\sigma_j^2) (h,v_j) v_j$. Differentiating the sum formally term-by-term, we obtain
		\begin{equation*}
			\sqrt{2}\sum_{j=1}^\infty \pi j \varphi_s(\sigma_j^2) (h,v_j) \cos\left(\pi j \cdot \right).
		\end{equation*}
				We now show that this series converges uniformly in $x$. Indeed, using Cauchy-Schwartz,
				\begin{align*}
			&\sum_{j=1}^\infty \pi j \varphi_s(\sigma_j^2)|(h,v_j)| |\cos(j\pi x)| \\
            &\qquad \le \pi \sqrt{\sum_{j=1}^\infty (h,v_j)^2} \sqrt{\sum_{j=1}^\infty j^2 \varphi_s^2(\sigma_j^2)}\le \pi^{1+2s} \rho \sqrt{\sum_{j=1}^\infty j^{2-4s}},
		\end{align*}
				and the right hand side converges whenever $s>\frac{3}{4}$, uniformly in $x$. Consequently, it holds that
		\begin{equation*}
			(f^\dagger)' =  \sqrt{2}\sum_{j=1}^\infty j \pi \varphi_s(\sigma_j^2) (h,v_j) \cos(\pi j \cdot).
		\end{equation*}
				Similarly, differentiating $f^\dagger$ twice formally term-by-term, we get
				\begin{equation*}
			- \sqrt{2}\sum_{j=1}^\infty j^2 \pi^2 \varphi_s(\sigma_j^2)(h,v_j) v_j(\cdot),
		\end{equation*}
				and
				\begin{align*}
			\sum_{j=1}^\infty \pi^2j^2 \varphi_s(\sigma_j^2)|(h,v_j)| |v_j(x)|&\le  \pi^2\sqrt{\sum_{j=1}^\infty (h,v_j)^2} \sqrt{\sum_{j=1}^\infty j^4 \varphi_s^2(\sigma_j^2)}\le \pi^{2+2s} \rho \sqrt{\sum_{j=1}^\infty j^{4-4s}},
		\end{align*}
				where the right hand side converges uniformly in $x$ whenever $s>\frac{5}{4}$.
			\end{proof}

\section{Numerical experiments}
\label{sec:num}
This section presents three numerical demonstrations of GCV. We first apply the
method to the integral equation \eqref{s1:e1}, continue with a Gaussian‐blur
deblurring problem, and conclude with a computed-tomography (CT) experiment.

\paragraph{Integral Equation \eqref{appl:int}}
Throughout the simulations we fix $D=2^{14}=16384$ and draw i.i.d. standard
Gaussian random variables $X_j$, $j=1,...,D$. From these realizations we
construct three exact solutions
\begin{equation*}
	f^{i,\dagger}:=\sum_{j=1}^D \sigma_j^{s_i}  X_jv_j
\end{equation*}
where $s_i \in\left\{\frac{1}{4},\frac{3}{4},\frac{5}{4}\right\}$ controls the
smoothness of the solution. We define the corresponding exact data as
\begin{align*}
	g^{i,\dagger}_m:&= \left(K f^{i,\dagger}(\xi_{l,m})\right)_{l=1}^m = \sqrt{2} \left(\sum_{j=1}^D \left(j \pi\right)^{2(s_i+1)} X_j \sin\left(j\pi \xi_{l,m}\right)\right)_{l=1}^m\in\R^m.
\end{align*}
Perturbed data are generated via
\begin{equation}\label{num:meas}
	g^{i,\delta}_m:= g^{i,\dagger}_m  +\delta \begin{pmatrix} Z_1 \\ ... \\ Z_m\end{pmatrix},
\end{equation}
with $Z_1,...,Z_m$ i.i.d. standard Gaussian, sampled anew in every simulation
loop. We first give formulas to calculate the error of our estimator. Using
\eqref{claim:sv}, the projection $(f^{i,\dagger},v_{k,m}) =  \sum_{j=1}^D
\sigma_j^{s_i+1} X_j (v_j,v_{k,m})$
can be calculated exactly for $k=1,...,m$; define
$f^{i,\dagger}_m:=\sum_{j=1}^m(f^{i,\dagger},v_{j,m})v_{j,m}$. Then
\begin{align*}
	\|f_{k,m}^\delta - f^{i,\dagger}_m\|^2 = \sum_{j=1}^k \left(\frac{(g^{i,\delta}_m,u_{j,m})_{\R^m}}{\sigma_{j,m}} - (f^{i,\dagger},v_{j,m})\right)^2 + \sum_{j=k+1}^m (f^{i,\dagger},v_{j,m})^2
\end{align*}
and
\begin{align*}
	f^{i,\dagger}_m-f^{i,\dagger}&= \sum_{j=1}^m(f^{i,\dagger},v_{j,m}) v_{j,m} - \sum_{l=1}^D (f^{i,\dagger},v_l) v_l\\
	&=\sum_{l=1}^D \left(\sum_{j=1}^m (f^{i,\dagger},v_{j,m})(v_{j,m},v_l) - (f^{i,\dagger},v_l)\right)v_l \\
    & \quad +\sum_{l=D+1}^\infty \sum_{j=1}^m(f^{i,\dagger},v_{j,m})(v_{j,m},v_l) v_l.
\end{align*}
By orthogonality ($\|f^{i,\delta}_k-f^{i,\dagger}\|^2 = \|f^{i,\delta}_k -
f^{i,\dagger}_m\|^2 +\|f^{i,\dagger}_m- f^{i,\dagger}\|^2$),
\begin{align*}\notag
	&\|f^\delta_{k,m} - f^{i,\dagger}\|^2\\\notag
	=~& \sum_{j=1}^k\left(\frac{(g^{i,\delta}_m,u_{j,m})_{\R^m}}{\sigma_{j,m}} - (f^{i,\dagger},v_{j,m})\right)^2 + \sum_{j=k+1}^m(f^{i,\dagger},v_{j,m})^2\\\label{num:err1}
	&\quad + \sum_{j=1}^D\left(\sum_{j=1}^m(f^{i,\dagger},v_{j,m})(v_{j,m},v_l) - (f^{i,\dagger},v_l)\right)^2 + \sum_{l=D+1}^\infty\left(\sum_{j=1}^m(f^{i,\dagger},v_{j,m})(v_{j,m},v_l)\right)^2
\end{align*}
and we define, suppressing the dependence on $\delta$ and $m,i$, the approximative error of the estimator:
\begin{align}
	e_k:&= \left(\sum_{j=1}^k\left(\frac{(g^{i,\delta}_m,u_{j,m})_{\R^m}}{\sigma_{j,m}} - (f^{i,\dagger},v_{j,m})\right)^2 + \sum_{j=k+1}^m(f^{i,\dagger},v_{j,m})^2\right.\\
	&\qquad\left. +
	\sum_{j=1}^D\left(\sum_{j=1}^m(f^{i,\dagger},v_{j,m})(v_{j,m},v_l) -
	(f^{i,\dagger},v_l)\right)^2\right)^\frac{1}{2}.\label{num:err}
\end{align}
In the simulations we calculate the computable GCV estimator
\begin{equation}\label{num:gcv}
	k_{\rm {gcv}}:=\arg\min_{0\le k\le \frac{m}{2}} \frac{\sum_{j=k+1}^m(g^{i,\delta}_m,u_{j,m})_{\R^m}}{(1-\frac{k}{m})^2},
\end{equation}
and the in practice unfeasible optimal estimator
\begin{equation}\label{num:opt}
	k_{\rm opt}:=\arg\min_{0\le k \le m} e_k,
\end{equation}
for reference. The error we make in approximating $\|f^\delta_{k,m}-f^\dagger\|$ by \eqref{num:err} can be bounded from above as follows (where expectation is with respect to the $X_j's$):
\begin{align*}
	&\E\left[\left| e_k^2 - \|f^{i,\delta}_k-f^{i,\dagger}\|^2\right|\right]\\
	 =~&\sum_{l=D+1}^\infty\E\left[\left(\sum_{j=1}^m\sigma_j^{s_i} X_j(v_{j,m},v_l)\right)^2\right] = \sum_{l=D+1}^\infty \sum_{j=1}^m \sigma_j^{2s_i} (v_{j,m},v_l)^2\\
	\le~& \sum_{l=D+1}^\infty \max_{j=1,...,m} \sigma_j^{2s_i}(m+1) \frac{\sigma_l^2}{\sigma_{j,m}^2}\le 3\max_{j=1,...,m} \sigma_j^{2s_i - 2} \sum_{l=D+1}^\infty \sigma_l^2 \le \frac{3}{\pi^4} \frac{1}{D^3} \max_{j=1,...,m} \sigma_j^{2s_i-2}
\end{align*}
and so
\begin{align*}
	\delta_i^2:= \frac{3}{\pi^4}\begin{cases} \frac{(m\pi)^3}{D^3} &,\quad \mbox{for  } s_i=\frac{1}{4}\\
		\frac{m\pi}{D^3}    &,\quad \mbox{for   } s_i=\frac{3}{4}\\
		\frac{1}{\pi D^3} &,\quad \mbox{for    } s_i=\frac{5}{4}
	\end{cases}.
\end{align*}
is an upper bound for 	$\E\left[\left| e_k^2 - \|f^{i,\delta}_k-f^{i,\dagger}\|^2\right|\right]$. For our choices of $m$ and $D$ we thus obtain
\begin{equation*}
	\delta_i\asymp \begin{cases} 2^{-9}&,\quad \mbox{for   } s_i=\frac{1}{4}\\
		2^{-17}&,\quad \mbox{for   } s_i=\frac{3}{4}\\
		2^{-21}&,\quad \mbox{for   } s_i=\frac{5}{4}
	\end{cases}.
\end{equation*}
We will see below in the error plots that $\delta_i$ is of smaller order than  $e_k$ in all cases. We consider different noise levels $\delta$, which we determine implicitly via the signal-to-noise ratio (SNR). The SNR is defined as
\begin{equation*}
	{\rm SNR}:=\frac{\|{\rm signal}\|}{\|{\rm noise}\|} = \frac{\|g^{i,\dagger}\|_m}{\sqrt{m}\delta}.
\end{equation*}
For each exact solution $f^{i,\dagger}$ and each ${\rm SNR}$, we generate $200$ independent noisy measurements $g^{\delta}_m$ (in \eqref{num:meas}), and calculate $k_{\cdot}$ along with the corresponding errors $e_{k_{\cdot}}$, where $\cdot\in\{{\rm gcv},{\rm opt}\}$, see $\eqref{num:err}-\eqref{num:opt}$. We fix the number of measurements as $m=2^{9}$ and let ${\rm SNR}$ vary over $\{1,10,...,10^8\}$ (that is we effectively vary the noise level $\delta$).
The results are presented in Figure \ref{num:res}. In the left column we visualize the statistics as box plots and in the right column we give the corresponding sample means and sample standard deviations in tabular form. In each box plot, the upper and lower edge give the 75- respective 25\% quantile of the statistic $e_{k_{\cdot}}$ for $\cdot={\rm gcv}$ (red) and $\cdot={\rm opt}$ (blue). The median of the statistic is given as a red bar inside the boxes. The whiskers extend to the samples whose distance to the upper respectively lower edge is less than six times the height of the box. All samples which fall outside of the whiskers are plotted individually as red crosses (outliers). Outliers above the upper limit $1$ are plotted just above, retaining their relative order, but not given the exact value.

We clearly observe the convergence of the error, as the noise level decreases
(that is as the SNR increases). Hereby, the convergence rate of the generalized
cross-validation is comparable to the one of the optimal rate at least for
small noise levels. For larger noise levels (smaller SNR) the statistic for the
generalized cross-validation is rather spread out. Moreover we observe
saturation of the error for rougher solutions with smoothness parameter
$s_i\in\{1/4,3/4\}$, due to a dominating discretization error. The difference
between $e_{k^\delta_{\rm gcv}}$ and $e_{k^\delta_{\rm opt}}$ in the saturation
regime is due to the constraint $k^\delta_{\rm gcv}\le \frac{m}{2}$. Note that
in all cases the error for the largest SNR is still of higher order than the
errors $\delta_i$ we make in the approximation.

\begin{figure}
	\centering
	\begin{minipage}[c]{0.49\textwidth}
		\includegraphics[width=\textwidth]{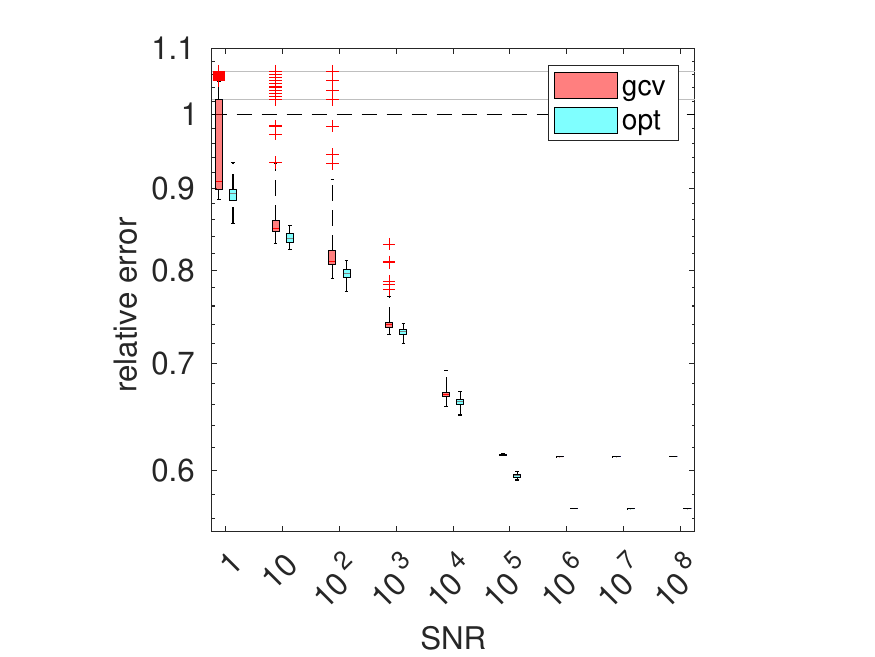}
	\end{minipage}
	\begin{minipage}[c]{0.49\textwidth}
		\centering
		\setlength{\tabcolsep}{4pt}
	\begin{tabular}{c|cc|}
			${\rm SNR}$ & $e_{\rm gcv}$  & $e_{\rm opt}$  \\
			\toprule
			1&	1.2e0 $\pm$ 8.3e-1	& 8.9e-1 $\pm$ 1.1e-2    \\
			10&	9.8e-1 $\pm$ 1.2e0	& 8.4e-1 $\pm$ 6.9e-3    \\
			$10^2$&	8.3e-1 $\pm$ 7.1e-2	& 8.0e-1 $\pm$ 6.9e-3    \\
			$10^3$&	7.4e-1 $\pm$ 1.2e-2	& 7.3e-1 $\pm$ 4.1e-3    \\
			$10^4$&	6.7e-1 $\pm$ 4.6e-3	& 6.6e-1 $\pm$ 3.6e-3    \\
			$10^5$&	6.1e-1 $\pm$ 2.4e-4	& 6.0e-1 $\pm$   1.4e-4  \\
			$10^6$&	6.1e-1 $\pm$ 1.9e-6	& 5.7e-1 $\pm$ 3.0e-5    \\
			$10^7$&	6.1e-1 $\pm$ 2.1e-8	& 5.7e-1 $\pm$ 3.6e-7    \\
			$10^8$&	6.1e-1 $\pm$ 8.3e-10	& 5.7e-1 $\pm$ 1.5e-8    \\
			\bottomrule
		\end{tabular}
	\end{minipage}
	\begin{minipage}[c]{0.49\textwidth}
		\includegraphics[width=\textwidth]{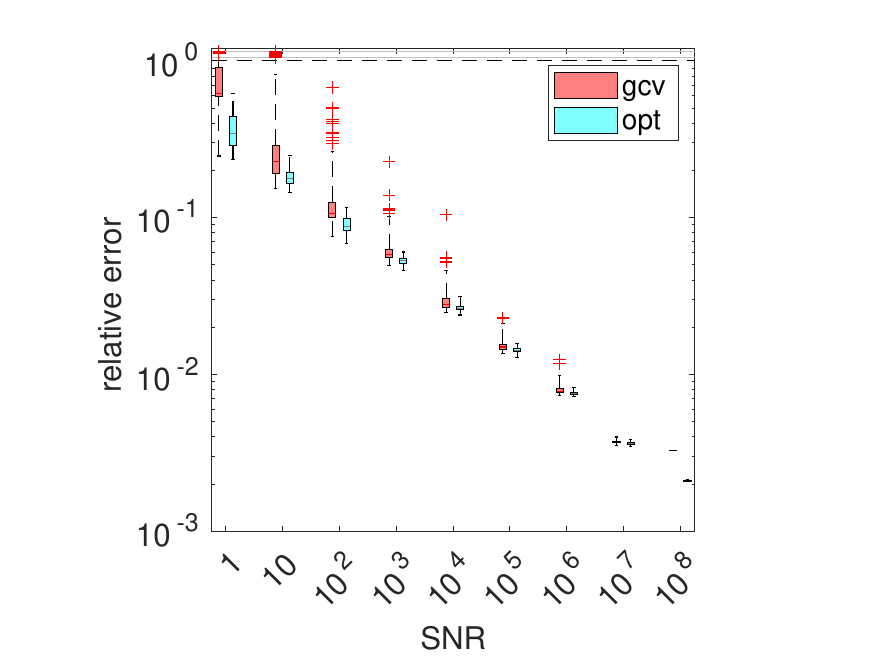}
	\end{minipage}
	\begin{minipage}[c]{0.49\textwidth}
		\centering
		\setlength{\tabcolsep}{4pt}
	\begin{tabular}{c|cc|}
			${\rm SNR}$ & $e_{\rm gcv}$  & $e_{\rm opt}$  \\
			\toprule
			1&	2.5e0 $\pm$ 1.0e1	& 3.8e-1 $\pm$ 1.1e-1    \\
			10&	3.7e-1 $\pm$ 5.9e-1	& 1.8e-1 $\pm$ 2.3e-2    \\
			$10^2$&	1.3e-1 $\pm$ 7.7e-2	& 9.0e-2 $\pm$ 1.0e-2    \\
			$10^3$&	6.4e-2 $\pm$ 1.8e-2	& 5.3e-2 $\pm$ 2.8e-3    \\
			$10^4$&	3.0e-2 $\pm$ 7.3e-3	& 2.7e-2 $\pm$ 1.2e-3    \\
			$10^5$&	1.5e-2 $\pm$ 1.3-3	& 1.4e-2 $\pm$ 5.1e-4    \\
			$10^6$&	7.9e-3 $\pm$ 5.5e-4	& 7.6e-3 $\pm$ 1.7e-4    \\
			$10^7$&	3.7e-3 $\pm$ 7.3e-5	& 3.6e-3 $\pm$ 5.8e-5    \\
			$10^8$&	3.3e-3 $\pm$ 6.0e-7	& 2.1e-3 $\pm$ 1.5e-5    \\
			\bottomrule
	\end{tabular}
	\end{minipage}
	\begin{minipage}[c]{0.49\textwidth}
		\includegraphics[width=\textwidth]{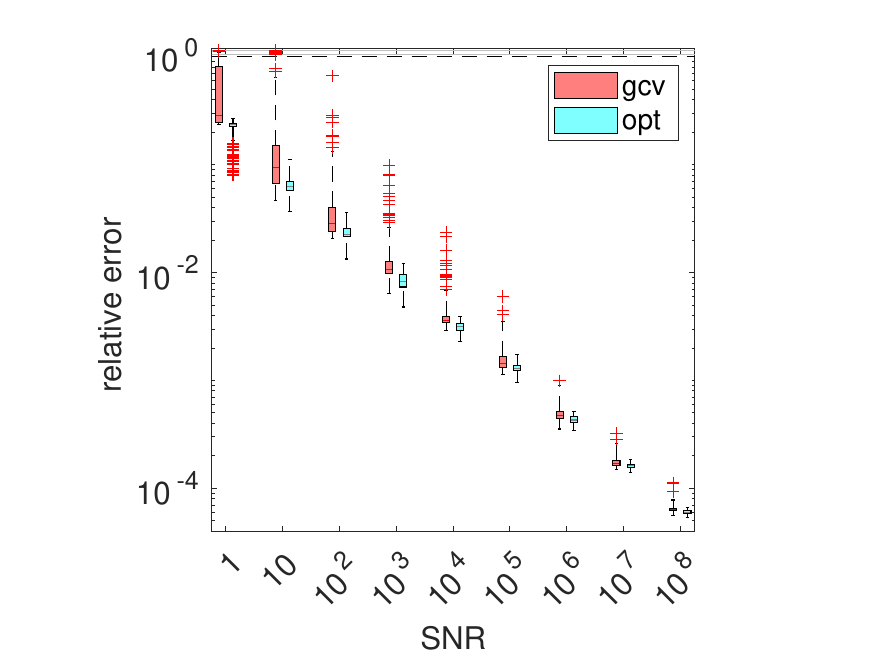}
	\end{minipage}
	\begin{minipage}[c]{0.49\textwidth}
		\centering
		\setlength{\tabcolsep}{4pt}
	\begin{tabular}{c|cc|}
			${\rm SNR}$ & $e_{\rm gcv}$  & $e_{\rm opt}$  \\
			\toprule
			1&	1.5e0 $\pm$ 4.2e0	& 2.2e-1 $\pm$ 4.2e-2    \\
			10&	2.2e-1 $\pm$ 4.8e-1	& 6.6e-2 $\pm$ 1.4e-2    \\
			$10^2$&	4.7e-2 $\pm$ 6.1e-2	& 2.4e-2 $\pm$ 3.4e-3    \\
			$10^3$&	1.5e-2 $\pm$ 1.2e-2	& 8.4e-3 $\pm$ 1.5e-3    \\
			$10^4$&	4.4e-3 $\pm$ 2.6e-3	& 3.1e-3 $\pm$ 3.0e-4    \\
			$10^5$&	1.6e-3 $\pm$ 5.6e-4	& 1.3e-3 $\pm$ 1.2e-4    \\
			$10^6$& 4.9e-4	$\pm$ 8.0e-5 & 4.3e-4 $\pm$ 3.5e-5    \\
			$10^7$&	1.8e-4 $\pm$ 2.1e-5	& 1.6e-4 $\pm$ 7.9e-6    \\
			$10^8$&	6.4e-5 $\pm$ 5.3e-6	& 6.0e-5 $\pm$ 2.6e-6    \\
			\bottomrule
				\end{tabular}
	\end{minipage}
	\caption{Left column: Boxplots of the errors  for 200 independent runs, with different signal-to-noise ratios (SNR). Right column: The corresponding sample mean and sample standard deviation of the errors. First row: rough solution. Second row: differentiable solution. Third row: twice differentiable solution.}
	\label{num:res}
\end{figure}

\paragraph{Image Deblurring}
The second experiment investigates a practical image deblurring problem. The task
is to recover a sharp image $f(x,y)$ from its blurred observation $g(x,y)$. The
process is expressed as solving an integral equation
\begin{equation*}
    \label{eq:intblur}
    g(x,y) = \iint_{\mathbb{R}^2} f(u,v) k(x-u,y-v)du\ dv,
\end{equation*}
where $k$ is a blur kernel, a Gaussian point-spread function is employed:
%
\begin{equation}
    \label{eq:gaussblur}
    k(x,y) = G_{\sigma}(x,y) = \frac{1}{2\pi \sigma^2} \exp(-\frac{x^2 +
    y^2}{2\sigma^2}),
\end{equation}
where $\sigma > 0$ is the standard deviation of the Gaussian and controls the
amount of smoothing. The convolution matrix is derived by discrete
approximation of \eqref{eq:gaussblur}: we sample the Gaussian function at
discrete points $(x,y)$ on a \textit{smaller} integer grid, taking into account
a desired kernel size $K \times K$. Thus, we derive each element $(m,n)$ in the
discrete \textit{smaller} kernel, also called convolution matrix:
\begin{equation*}
    w_{m,n} = G_{\sigma}(m,n) = \frac{1}{2\pi \sigma^2} \exp(-\frac{m^2 +
    n^2}{2\sigma^2}),
\end{equation*}
where $m,n \in \{-M,...,M\},\ M = \frac{K - 1}{2}$. Matrix entries $w_{m,n}
\leftarrow \frac{w_{m,n}}{w}$ are normalized by the factor:
    $$w = \sum_{m=-M}^{M} \sum_{n=-M}^{M} w_{m,n}.$$
The application of the blur to an image is represented as a convolution between
the image and a \textit{smaller} kernel. Expressing the convolution in a linear
formulation for each pixel $(i,j)$ of the image $f$ that we consider as an $H
\times W$ matrix of pixel values $f[i,j]$, we get:

\begin{equation}
  \label{eq:convolution}
  g[i,j] = (f*G_{\sigma})[i,j] = \sum_{m=-M}^{M}\sum_{n=-M}^{M}\,f[i-m,\,j-n]\,
  w_{m,n}.
\end{equation}

For notational convenience we vectorise each image into a column vector $u$ by
stacking its rows. For instance, $f$ derives its vector-image $u$ by
\begin{equation}
    \label{eq:vectorize}
    u[\alpha] = f[i,j], \qquad \alpha = iW + j \; \Leftrightarrow \;
    \begin {cases}
    & i \;=\; \left\lfloor \frac{\alpha}{W} \right\rfloor, \\
    & j \;=\; \alpha \bmod W.
    \end{cases}
\end{equation}
Since convolution is a linear operation in $f$, there exists a matrix
$A \in \mathbb{R}^{HW \times HW}$ that acts on its vectorized image $u$ and turning it into a blurred vector-image $b = Au$.
Note that on one hand, each row $A[\alpha,\: \cdot\ ]$ corresponds to one
output pixel $b[\alpha] = A[\alpha,\: \cdot\ ]u$, and, on the other hand,
$b[\alpha] = g[i,j] = (f * G_{\sigma})[i, j]$, where $\alpha = iW + j$ by
definition. Furthermore, for each term $f[i - m, j - n]$ in
\ref{eq:convolution}, its location in $u$ is $\beta = (i - m)W + (j - n)$ by
\ref{eq:vectorize}. Thus,
\begin{equation}
  \label{eq:blurred}
  b[\alpha] = \sum_{m=-M}^{M}\sum_{n=-M}^{M} u[(i - m) W + (j -  n)]\ w_{m,n},
  \quad \alpha = iW + j.
\end{equation}
We build the matrix $A$ row by row and thus for each row $A[\alpha,\: \cdot\ ]$:
\begin{enumerate}
  \item compute $i$ and $j$ from $\alpha$ by \ref{eq:vectorize},
  \item for each kernel offset $(m,n) \in \{-M,...,M\}^2$ calculate $\beta = (i
  - m)W + (j - n)$,
  \item if $(i-m, j-n)$ is within the valid range $[0,H-1] \times [0,W-1]$, set
  $A[\alpha, \beta] = w_{m,n}$, in accordance with \ref{eq:blurred},
  \item otherwise, set $A[\alpha, \beta] = 0$ (or the other value corresponding
  to a chosen boundary condition).
\end{enumerate}
Thereby, $A[\alpha, \beta]$ is the coefficient that multiplies $u[\beta]$ to
produce $b[\alpha]$ after summation:
\begin{equation}
  A[\alpha, \beta]
  \;=\;
  \begin{cases}
    w_{m,n}, & \text{if } \beta = (i - m)\,W + (j - n), \;\text{and }
    i = \left\lfloor \frac{\alpha}{W} \right\rfloor,\
    j = \alpha \bmod W,\\[6pt]
    0 , & \text{(or a boundary condition value) otherwise}.
  \end{cases}
\end{equation}

Thus, the operator $A$ takes each pixel of $x$, multiplies their values
by weight given by a Gaussian kernel and, sum the results deriving blurred
vector-image $b$.

To generate a task for our example, we use the frameworks IR Tools
\cite{gazzola2019irtools} and AIR Tools II \cite{hansen2018airtoolsii}. This
software is designed for algebraic iterative reconstruction methods, tailored
to solve regularized inverse problems. We use its capabilities in building a
blurring operator $A$ and construct an inverse problem of deducting the
original vector-image $x$ from  the blurred one $b$. We use default IR Tools
settings to generate data for use in image deblurring problems. Namely, we choose
an image with $H \times W = N \times N = 256 \times 256$ resolution,
\textit{medium} blur severity, corresponding to $\sigma = 4$, kernel size
$K = 256$ spanning the entire image,
and \textit{reflective} (or Neumann, or reflexive) boundary conditions. Note that
in the default IR Tools settings an enlarged (padded) test image is first
blurred using \textit{zero} boundary conditions. After blurring, a central
subimage of size $N$ is extracted from both the original and the blurred
images. This method ensures that no inverse crime is committed and $A
\times x \approx b$, since the situation when $A \times x = b$ exactly is
unrealistic. This is how IR Tools simulates discretization error.
We repeat the experiment $m = 1000$ times and vary SNR over $\{10^{-3}, 10^{-2},
..., 10^2, 10^3\}$ for noise levels added to the blurred image-vector
$b$.

Figure \ref{fig:blurplot} demonstrates the results for our experiments in
blue and red boxes for the GCV method and optimal solution. Edges, whiskers and
outliers settings are similar to the ones in the Figure \ref{num:res}, the
median values are blue or red bars inside the boxes. The figure demonstrates
convergence of the error depending on noise level decrease. Again, the
convergence rate for GCV is comparable to optimal. Note that for results not
shown with even smaller
noise (SNR $> 10^3$) the domination of the discretization error became clear and
lead to saturation of the error.
In the meantime, the Figure \ref{fig:blurplot} demonstrates that alike the
relative error, the truncation index $k_{\rm gcv}$ is close to the optimal one $k_{\rm opt}$.

Figure \ref{fig:blurpic} showcases the outcome of the generalized
cross-validation method applied to a picture of the Hubble space telescope
blurred with the Gaussian blur.

\begin{figure}
  \centering
  \begin{minipage}[c]{0.445\textwidth}
    \includegraphics[width=\textwidth]{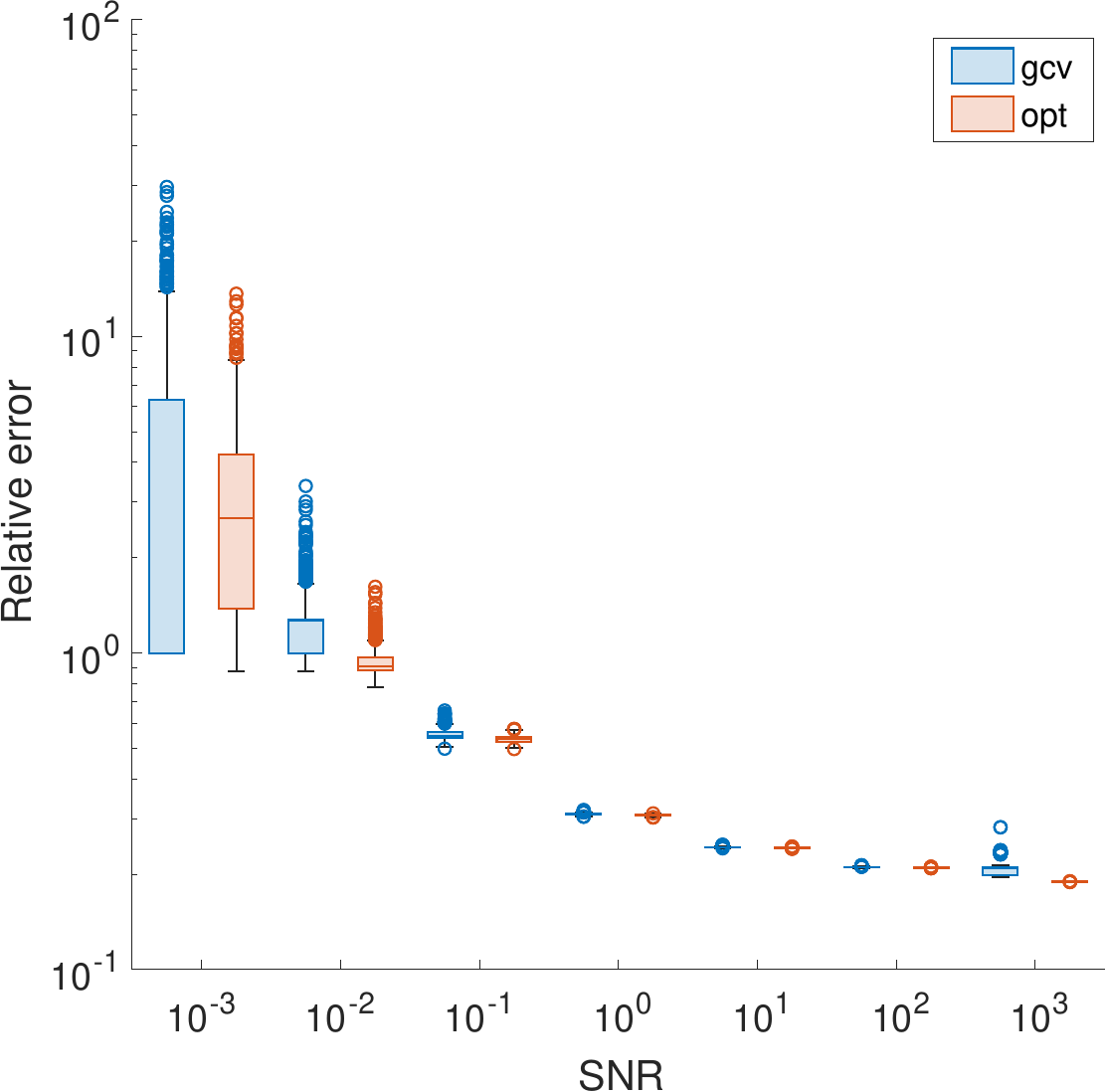}
  \end{minipage}
  \begin{minipage}[c]{0.545\textwidth}
    \includegraphics[width=\textwidth]{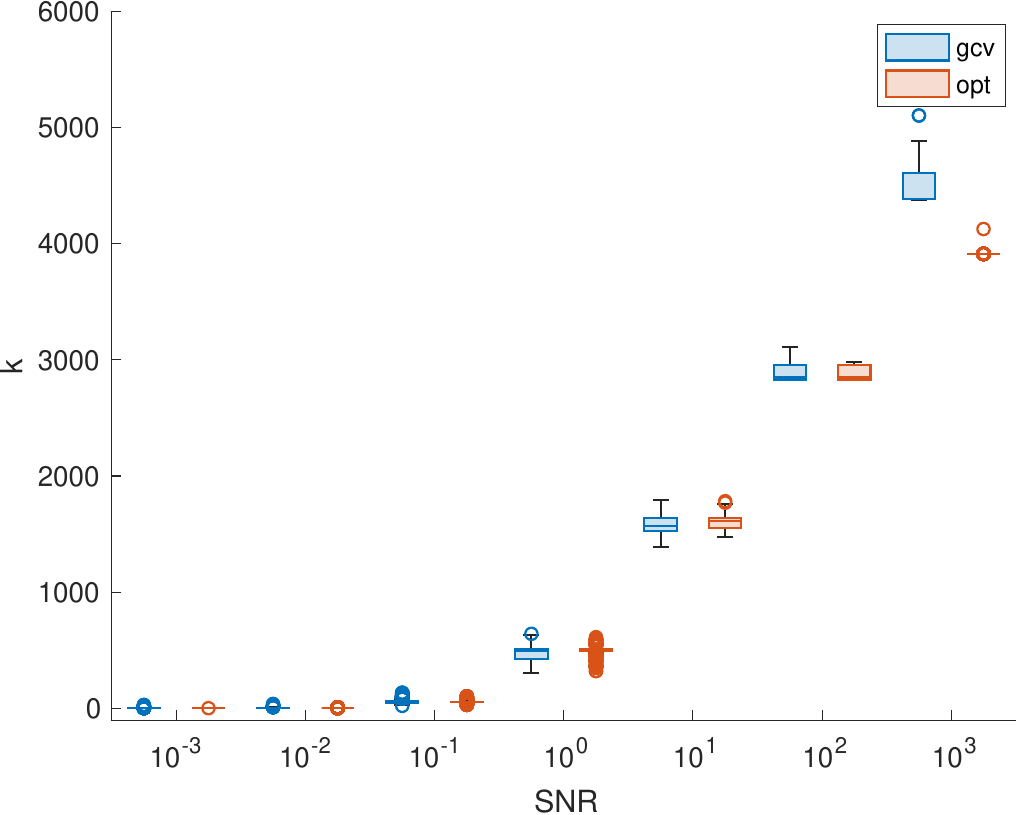}
  \end{minipage}

  \caption{Box-plot summary for 1000 independent runs, with different
  signal-to-noise ratios (SNR) with default IR Tools settings. On the left:
  relative errors, on the right: truncation index $k$ for GCV (blue) and optimal
  (red) solutions.}
  \label{fig:blurplot}
\end{figure}

\begin{figure}
  \centering
  \begin{minipage}[c]{0.49\textwidth}
    \includegraphics[width=\textwidth]{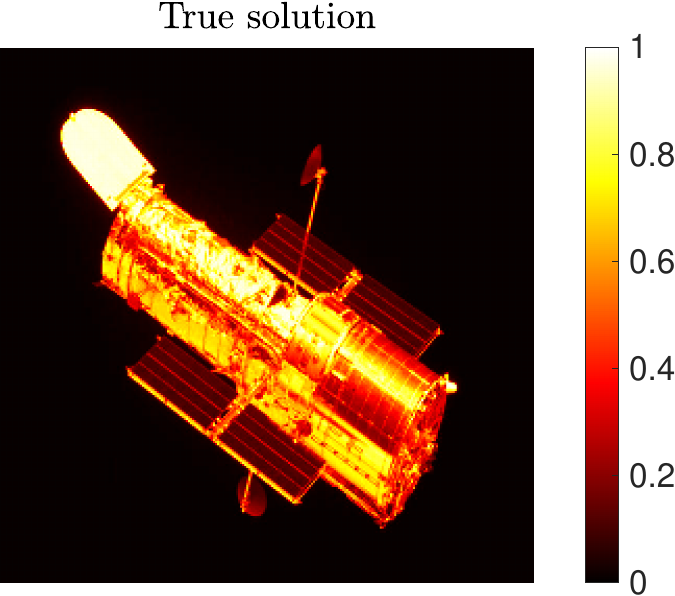}
  \end{minipage}
  \begin{minipage}[c]{0.49\textwidth}
    \includegraphics[width=\textwidth]{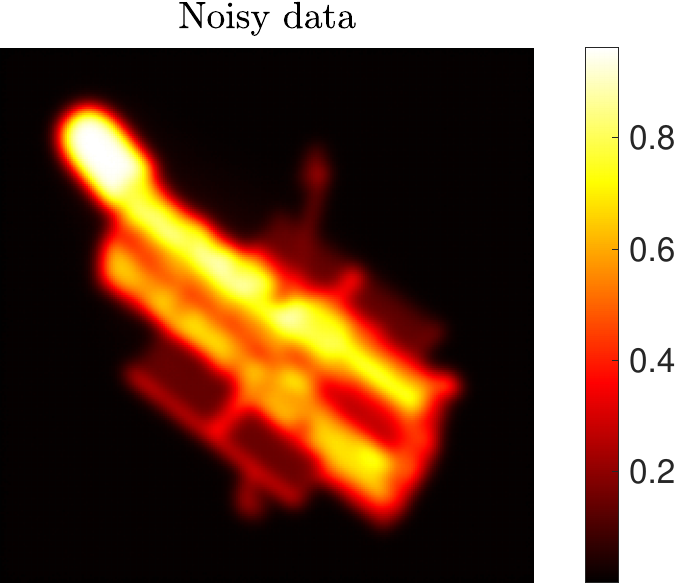}
  \end{minipage}
  \begin{minipage}[c]{0.49\textwidth}
    \includegraphics[width=\textwidth]{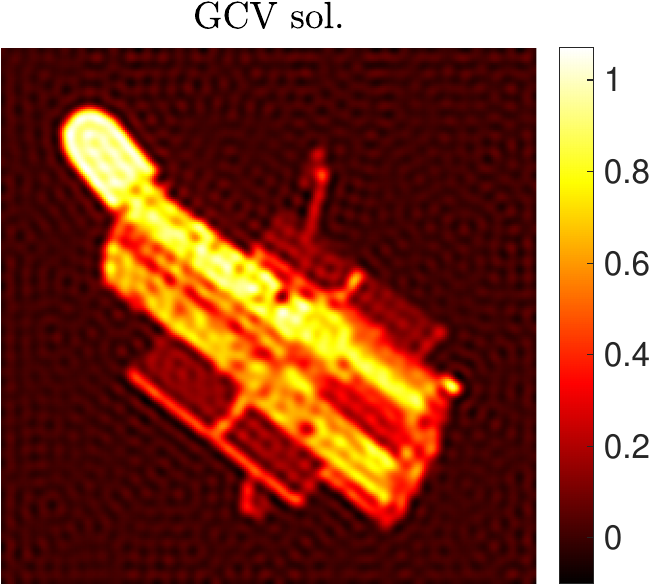}
  \end{minipage}
  \begin{minipage}[c]{0.49\textwidth}
    \includegraphics[width=\textwidth]{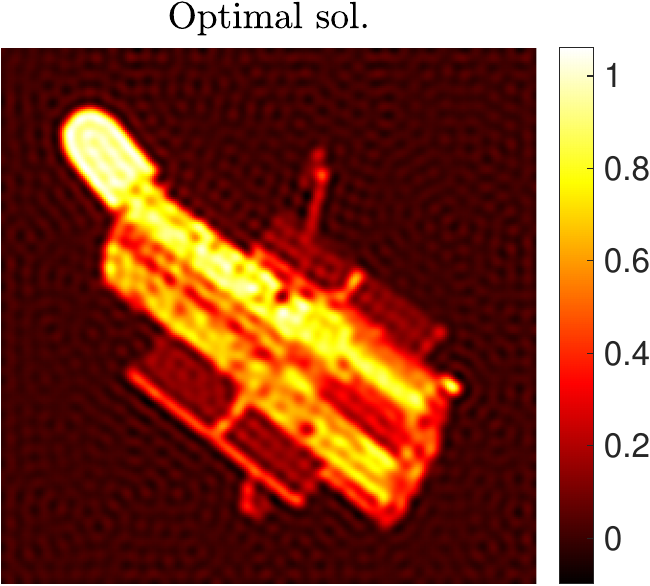}
  \end{minipage}

  \caption{Image deblurring test data for the example in Section \ref{sec:num}
  with SNR $= 10^2$. Top left: true image. Top right: noisy blurred data.
  Bottom left: restored image with generalized cross-validation method with the
  related error equals 0.19671 on $k = 3702$. Bottom left: optimal solution
  with the related error equals 0.195295 on $k = 3682$.}
  \label{fig:blurpic}
\end{figure}

The sparse matrix $A$ is not stored explicitly, instead, the PSF object provided
by IR Tools is used for on-the-fly computations. The implementation exploits the
built-in function
of the package and periodic properties of the task to calculate singular values
from PSF directly avoiding building the whole matrix. Since a default setting
for boundary conditions is \textit{reflective}, we utilize a discrete cosine
transform matrix to derive unique singular vectors saving computation time.

\paragraph{Computed Tomography}
The third study examines parallel-beam computed tomography (CT).
X-rays are photon beams that travel along (by assumption) straight lines through
tissue. Consider a monochromatic beam with intensity $I = I(t)$ moving from
source $t = 0$ to detector $t = L$ along straight path $\Gamma$ with arc length
parametrization $\textbf{x} = \textbf{x}(t)$. Intensity $I$ decreases along a
short path of length $\textup{d}t$ on $\Gamma$ at $\textbf{x}(t)$, and
proportional to $\textup{d}t$, as photons are absorbed according to Beer’s
law, that also states its loss of intensity $\textup{d}I$
\begin{equation}
  \label{eq:beerslaw}
  \textup{d}I = -f I \textup{d}t,
\end{equation}
where $f = f(\textbf{x})$ is the nonnegative attenuation coefficient -- a measure
closely related to the mass density of the material. In X-ray tomography, one
reconstructs this $f(\textbf{x})$ to visualize how the material absorbs photons.
Upon integrating from the source to the detector, the equation
$\ref{eq:beerslaw}$ implies
\begin{equation*}
  I(L) = I(0) \exp(-\int_{\Gamma} f(\textbf{x}) \textup{d}t).
\end{equation*}
The exponent’s integral, $$-\int_{\Gamma} f(\textbf{x}) \textup{d}t$$ is
precisely the line integral of $f$ along the ray $\Gamma$. In parallel-beam
CT one measures these line integrals of $f$. Moreover, since both the source and
detector rotate around the object and emit/receive a bundle of parallel rays, one
measures a set of the integrals for different positions. Mathematically, for a
two dimensional cross-section, each measurement is
\begin{equation}
  (Rf)(\theta, s)  = \int_{\ell(\theta,s)} f(x,y) ds_{\ell} =
  \int_{-\infty}^{\infty} \int_{-\infty}^{\infty} f(x,y) \delta
  \bigl(x\cos\theta + y\sin\theta - s\bigr) dx dy,
\end{equation}
for various rotational angles $\theta$ and distances from the origin $s$
characterizing the shift from source. Thereby, the attenuation coefficient
$f(x,y)$ is integrated along the line $\ell(\theta,s)$ with its small element
$ds_{\ell}$.
In the integral on the right, $\delta(\cdot)$ stands for the Dirac delta
function. The resulting function $(Rf)(\theta, s)$ is called Radon transform of
$f(x,y)$.

We act similarly to Gaussian blur example and construct a matrix operator $A$,
that turns a test  image-vector $x$, derived from a two dimensional density image
$g$, into a signal-vector, called sinogram, of measurements $b$. Assume we have:
\begin{itemize}
  \item $N_{\theta}$ projection angles, indexed by $i_{\theta} =
  \{1,2,...,N_{\theta}\}$,
  \item for each angle, $N_s$ detector positions, indexed by $i_{s} =
  \{1,2,...,N_s\}$.
\end{itemize}
Therefore, the total number of measurements is $M = N_{\theta} \times N_s$ and
thus $b \in  \mathbb{R}^M$.

Similarly to the Gaussian blur example, each pixel $\beta$ of a vector-image
$x$ corresponds to some pixel $(i,j)$ of two dimensional image $g$. Since in
our model task we already have an original image $g$ for reference, we consider
the numerical values of color intensity $g[i,j]$ for every pixel $(i,j)$ as a
value of cross-section function $f$ within the area covered by the pixel. Each
row $\alpha$ of $A$ corresponds to one measured ray, each column
$\beta$ -- to one pixel. Therefore, $A[\alpha, \beta]$ is the weight that pixel
$\beta$ contributes to measurement $\alpha$. There are various choices for a
weighting kernel, e.g.,\ area-based or distance-based. We rely on the IR Tools
approach with length-base weights, i.e.,\ the weight of pixel $\beta$ equals to
the length of intersection between ray $\alpha$ and pixel $\beta$. Summarizing
it up with the property of $x[\beta]$ being an attenuation coefficient, we
derive for the $\alpha$-th measurement:
  $$b[\alpha] = \sum_{\beta = 1}^{HW} A[\alpha, \beta]\; x[\beta],$$
where $H\times W$ is an image $g$ resolution.

We build this example with the IR Tools default settings, such as image
resolution $H \times W = N \times N = 256 \times 256$, projection angles
$N_{\theta} = 180$ with $\theta = \{0,1,...,179\}$, number of rays $N_s =
\lfloor N\sqrt{2} \rceil$ for each source angle, where $\lfloor \cdot \rceil$
means rounding to the nearest integer value and the distance $d = N_s - 1$ from
the first ray to the last. Similarly to the previous example, IR Tools provides
data with no inverse crime committed and $A \times x \approx b$. Again, we take
$m = 1000$ number of measurements with SNR for signal-vector $b$ over $\{10^{-3},
10^{-2}, ..., 10^2, 10^3\}$.

Again, according to Figure \ref{fig:ctplot}, the relative error and the truncation index of
 generalized cross-validation is similar to the optimal choice.
Figure \ref{fig:ctpic} demonstrates the result of restoring the signal back to
the original reference model with GCV method comparing to the optimal solution.

\begin{figure}
  \centering
  \begin{minipage}[c]{0.515\textwidth}
    \includegraphics[width=\textwidth]{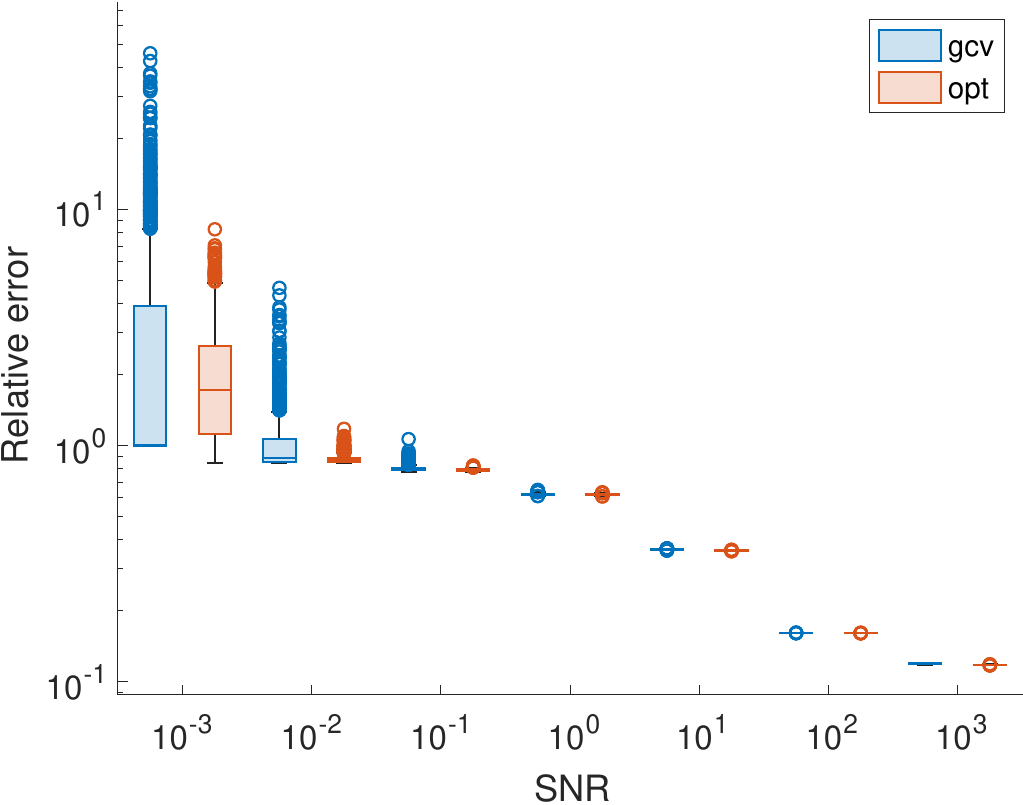}
  \end{minipage}
  \begin{minipage}[c]{0.465\textwidth}
    \includegraphics[width=\textwidth]{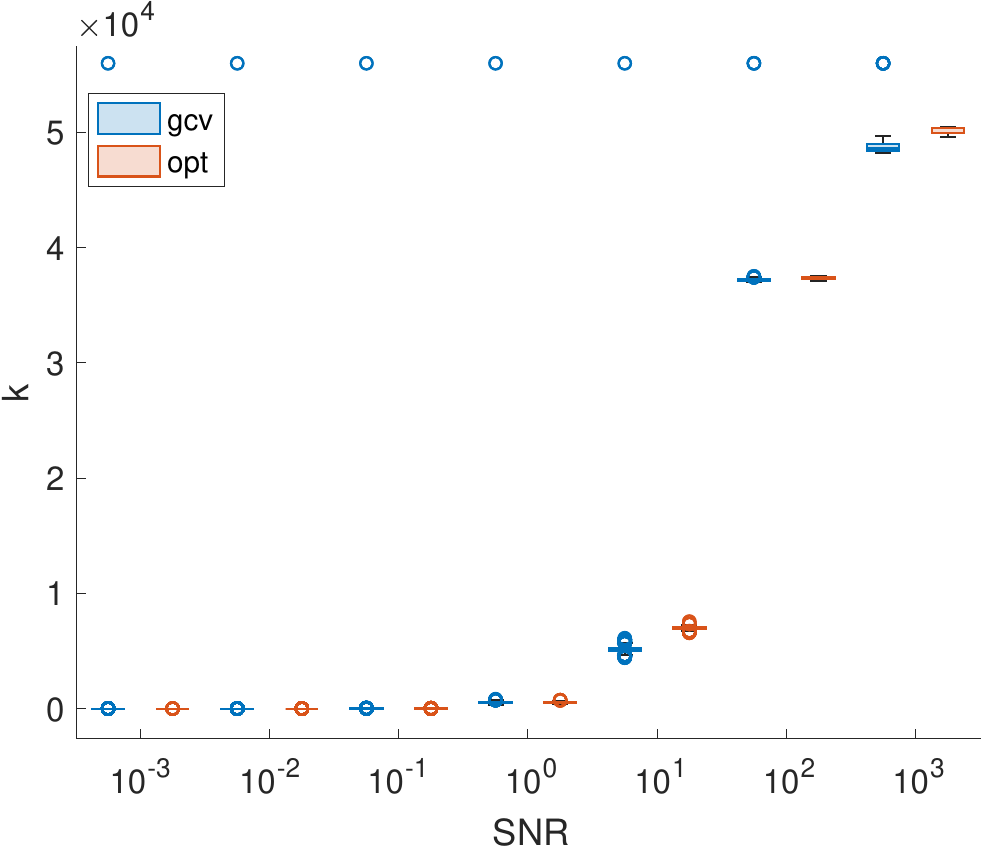}
  \end{minipage}

  \caption{Box-plot summary for 1000 independent runs, with different
  signal-to-noise ratios (SNR) with default IR Tools settings. On the left:
  relative errors, on the right: truncation index $k$ for GCV (blue) and optimal
  (red) solutions.}
  \label{fig:ctplot}
\end{figure}

\begin{figure}
  \centering
  \begin{minipage}[c]{0.32\textwidth}
    \includegraphics[width=\textwidth]{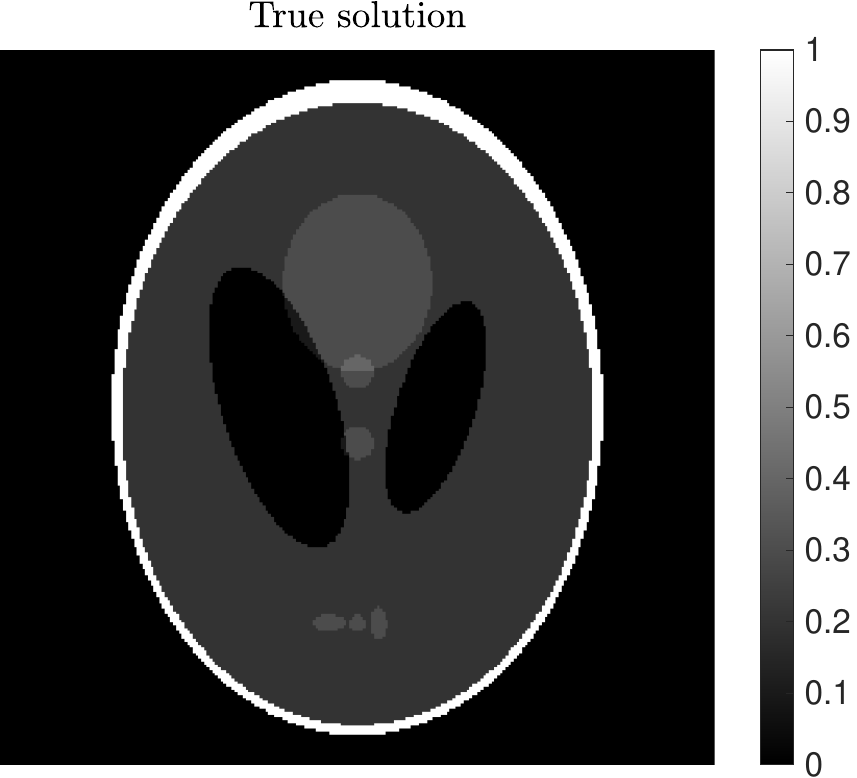}
  \end{minipage}
  \begin{minipage}[c]{0.32\textwidth}
    \includegraphics[width=\textwidth]{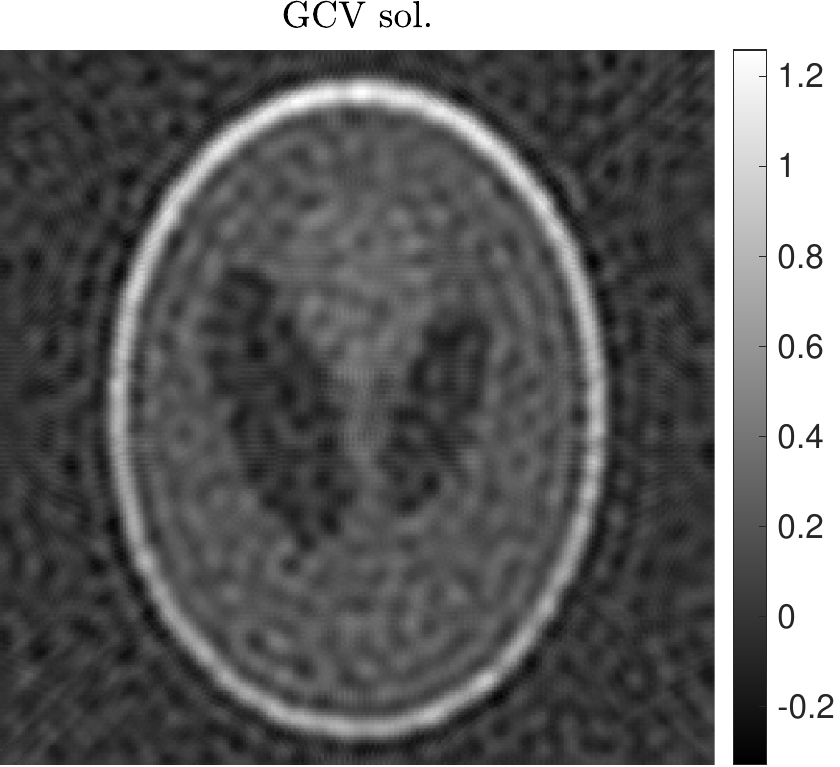}
  \end{minipage}
  \begin{minipage}[c]{0.32\textwidth}
    \includegraphics[width=\textwidth]{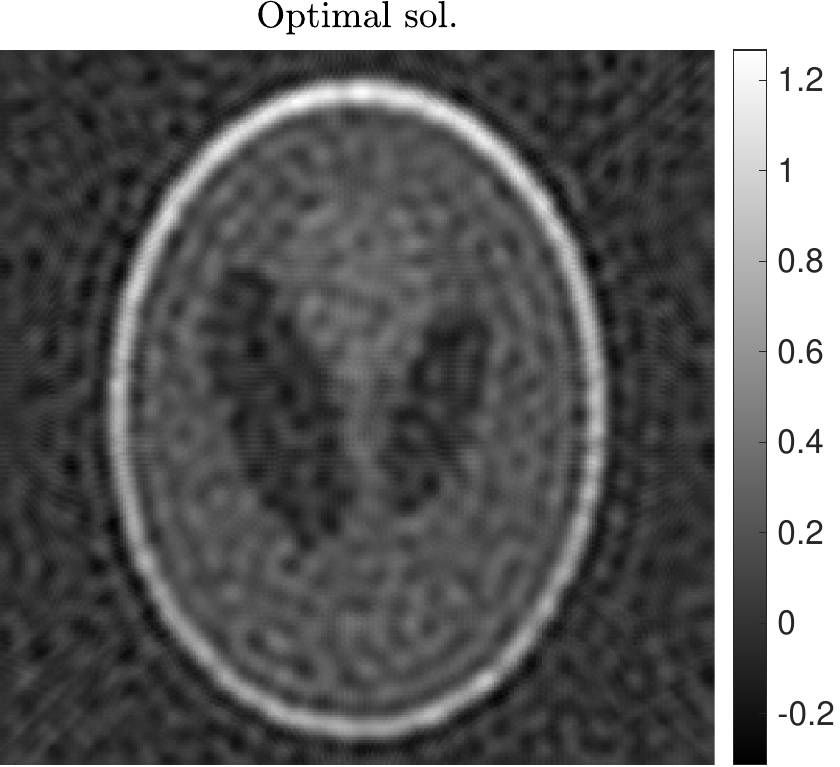}
  \end{minipage}

  \caption{Restoring a computed tomography signal for the example in Section
  \ref{sec:num}. Left: Shepp–Logan phantom image being a reference model, from
  what a signal is created. Middle: reconstruction obtained with GCV. Right:
  optimal solution.}
  \label{fig:ctpic}
\end{figure}

\section{Concluding remarks}
In this article we deduced rigorously a non-asymptotic error bound (in
probability) for GCV as a parameter choice rule for the solution of a specific
ill-posed integral equation. In particular we verified the optimality of the
rule in the mini-max sense, remarkably without imposing a self-similarity
condition onto the unknown solution, which up to our knowledge so far was
required for any rigorous and consistent optimality result for heuristic
parameter choice rules in the context of ill-posed problems. The
numerical evidence shows that the analytical guarantees are not of asymptotic
interest
but translate into robust, near-optimal reconstructions in image and signal
reconstruction pipelines.

\paragraph{Future work.} Several research directions remain open.
First, the findings from Theorem \ref{t0} could be extended to integral
equations with a general kernel $\kappa$. While Theorem \ref{l2} still is applicable, it remains to analyze the discretization error given by the
relation between the decomposition of the continuous operator $K$ and the
semi-discrete one $K_m$. In particular, the design matrix $T_m$ cannot be
calculated exactly in this case and has to be approximated by, e.g., a
quadrature rule, and the estimator should be based on the decomposition of the
quadrature approximation. Second, instead of spectral cut-off other
regularization methods, like Tikhonov regularization or some iterative scheme
should be considered. This will require non-trivial changes of the
probabilistic analysis of GCV. Finally, it would be interesting to extend the
analysis to other settings, for example non-parametric regression
based on kernelized spectral-filter algorithms.
\clearpage
\appendix
\section*{Appendix}
\addcontentsline{toc}{section}{Appendix}
\section{Proof of Lemma \ref{l1}}\label{sec:lem}
\begin{proof}[Proof of Lemma \ref{l1}]
		Recall that in our setting the kernel is the Green's function  of
		the Laplace equation, i.e., $(Kf)''=-f$. It is then straightforward to check
		that the solutions of the differential equation are eigenfunctions of $K$,
		which yield $\sigma_j$ and $v_j$.  While the discretization of the
		differential equation has been analyzed in detail, see, e.g.,
		\cite{boffi2010finite}, results for the corresponding
		discretization of the integral equation appear to be absent from the
		literature. We first show that
		the singular value decomposition of the semi-discrete $K_m$ is strongly
		related to the eigenvalue decomposition of the symmetric $m\times m$ matrix
		$\left(T_m\right)_{ij}:=\int \kappa(\xi_{i,m},y)\kappa(\xi_{j,m},y){\rm
		d}y=\frac{\xi_i(1-\xi_l)}{6}(-\xi_i^2-\xi_j^2+2\xi_j)$. Indeed, since $K_m^*
		\alpha =\sum_{j=1}^m \alpha_j \kappa(\xi_{j,m},\cdot)$ for $\alpha\in\R^m$,
		we obtain for $f_\alpha:=\sum_{j=1}^m \alpha_j\kappa(\xi_{j,m},\cdot)\in L^2$
		and $\lambda\in\R$ the relation
			\begin{align*}
K_m^*K_m f_\alpha = \lambda f_\alpha\qquad\Longleftrightarrow\qquad T_m \alpha = \lambda \alpha.
	\end{align*}
	 	 and consequently we need to find the eigenvalue decomposition of $T_m$. As auxiliary tools, we need the following  $m\times m$-dimensional symmetric matrices:
        \begin{align}\label{s1:e5}
	\Delta_m:&=\begin{pmatrix} 2 & -1 & ... && \\
			-1 &2 &-1 &...&\\
			\vdots &&&\ddots \\
			&&-1&2&-1\\
		& && -1 &2\end{pmatrix}, \qquad
		R_m:=\begin{pmatrix} 4 & 1 & ... && \\
			1 &4 &1 &...&\\
			\vdots &&&\ddots \\
			&&1&4&1\\
		& && 1 &4
    \end{pmatrix},
    \end{align}
    \begin{align*}
		S_m:=\begin{pmatrix}\kappa(\xi_{s,m},\xi_{t,m})\end{pmatrix}_{st}
	\end{align*}
		Note that $(m+1)^2\Delta_m$ is the discretization of the second derivative
		via centered second order finite differences on the homogeneous grid
		$\xi_{1,m},...,\xi_{m,m}$ and $\left(R_m\right)_{ij} =
		\frac{6}{m+1}(\Lambda_{i}^m,\Lambda_j^m)_{L^2(0,1)}$, with the hat functions
		$\Lambda_i(x):=(x-\xi_{i-1,m})(m+1) \chi_{(\xi_{i-1,m},\xi_{i,m}]}(x) +
		(\xi_{i+1,m}-x)(m+1)\chi_{(\xi_{i,m},\xi_{i+1,m}]}(x)$. We therefore show
		that $T_m$ and the matrices in  \eqref{s1:e5} have mutual eigenvectors
		\begin{equation}\label{s1:e5a}
		z_{k,m}:=\sqrt{\frac{2}{m+1}}\begin{pmatrix} \sin\left(\sqrt{\lambda_k} \xi_{1m}\right) & ... & \sin\left(\sqrt{\lambda_k}\xi_{mm}\right)\end{pmatrix}^T\in\R^m,
	\end{equation}
		with $k=1,...,m$. Using the polar identity $2i\sin(x) = e^{ix}-e^{-ix}$ and the closed-form expression for the partial geometric series with $q=e^\frac{ik\pi}{m+1}$, one sees that $\|z_{k,m}\|_{\R^m}^2=1$. By exploiting the polar identity again one easily verifies that $z_{k,m}$ are the eigenvectors of the circulant matrices $\Delta_m$ and $R_m$, and moreover that $\rho_{k,m}:=4+2\cos\left(\frac{k\pi}{m+1}\right)$ are the corresponding eigenvalues for $R_m$. Moreover, slightly lengthy but straightforward computations yield $\Delta_m S_m-S_m\Delta_m =0= \Delta_m T_m-T_m\Delta_m$, which implies that the $z_{k,m}$ are also the eigenvectors of $S_m$ and $T_m$. Next we show that the eigenvalues $\mu_{k,m}$ of $S_m$ are given by	$\mu_{k,m}=(-1)^{k+1}\cos\left(\frac{\sqrt{\lambda_k}}{2(1+m)}\right)\sin\left(\frac{\sqrt{\lambda_k}}{2(1+m)}\right)^{-1} \sin\left(\frac{\sqrt{\lambda_k}}{1+m}\right)^{-1}$. Using the polar identity for $q=e^\frac{ik\pi}{2(m+1)}$ and
	\begin{align*}
		\sum_{j=1}^m q^j j =\frac{q + q^{1 + m} (-1 - m + m q)}{(1 - q)^2}
	\end{align*}
		 yields
		\begin{align*}
		\sum_{l=1}^m \sin\left(\frac{\sqrt{\lambda_k} l}{m+1}\right) l=\frac{m+1}{2} (-1)^{k+1} \frac{\cos\left(\frac{\sqrt{\lambda_k}}{2(m+1)}\right)}{\sin\left(\frac{\sqrt{\lambda_k}}{2(m+1)}\right)},
	\end{align*}
	and because $\sin\left(k\pi m/(m+1)\right)=\sin\left(k\pi/(m+1)\right)$, the $\mu_{k,m}$ can be computed with the defining relation of the eigenvalues:
	\begin{align}\label{singe1}
		\mu_{k,m} \sin\left(\frac{\sqrt{\lambda_k} m}{m+1}\right)
		&=\sqrt{\frac{2}{m+1}}\mu_{k,m} (z_{k,m})_m =\sqrt{\frac{2}{m+1}} \left(S_m z_{k,m}\right)_m\\ &= \sum_{l=1}^m \xi_{l,m}(1-\xi_{m,m})\sin\left(\frac{\sqrt{\lambda_k} l}{m+1}\right)
		= \frac{(-1)^{k+1}}{2(m+1)}\frac{\cos\left(\frac{\sqrt{\lambda_k}}{2(m+1)}\right)}{\sin\left(\frac{\sqrt{\lambda_k}}{2(m+1)}\right)}.
	\end{align}
		To finally determine the eigenvalues of $\sigma_{k,m}^2$ of $T_m$ we set $w_{k,m}:=\sum_{l=1}^m (z_{k,m})_l \kappa(\xi_{l,m},\cdot)$ and normalize in two ways. First,
		\begin{align*}
	\|w_{k,m}\|^2&= \sum_{l,l'=1}^m (z_{k,m})_l (z_{k,m})_{l'} (\kappa(\xi_{l,m},\cdot),\kappa(\xi_{l',m},\cdot)) = z_{k,m}^T T_m z_{k,m}=\sigma_{k,m}^2.
	\end{align*}
		Second, expanding $\kappa(\xi_{j,m},\cdot)=\sum_{i=1}^m\kappa(\xi_{l,m},\xi_{i,m}) \Lambda_{i}(\cdot)$ in terms of the hat functions,
		\begin{align}\notag
		&\|w_{k,m}\|^2\\
		=&~\| \sum_{l=1}^m (z_{k,m})_l \sum_{i=1}^m \kappa(\xi_{l,m},\xi_{i,m}) \Lambda_i\|^2 = \|\sum_{i=1}^m \left( S_m z_{k,m}\right)_i \Lambda_i^m\|^2 = \mu_{k,m}^2\| \sum_{i=1}^m (z_{k,m})_i \Lambda_i^m\|^2\\\notag
		=&~\mu_{k,m}^2\sum_{i,i'=1}^m (z_{k,m})_i(z_{k,m})_{i'} \left(\Lambda_i^m,\Lambda_{i'}^m\right) = \mu_{k,m}^2\frac{1}{6(m+1)}\sum_{i=1}^m (z_{k,m})_i\left(R_m z_{k,m}\right)_i\\\label{singe2}
		=&~\mu_{k,m}^2\frac{4+2\cos\left(\frac{\sqrt{\lambda_k}}{m+1}\right)}{6(m+1)}.
	\end{align}
	Putting \eqref{singe1} and \eqref{singe2} together, using $\sin(2x)=2\sin(x)\cos(x)$ and $\cos(2x)=1-2\sin^2(x)$, then yields the explicit formulas for the eigenvalues $\sigma_{k,m}$ and the left singular functions $v_{k,m}$. Finally, we calculate the right singular vectors $u_{k,m}$:
\begin{align*}
	(u _{k,m})_j&= \frac{1}{\sigma_{k,m}}  (K_m v_{k,m})(\xi_{j,m}) = \frac{1}{\sigma_{k,m}} \sum_{l=1}^m (z_{k,m})_l (K_m \kappa(\xi_{l,m},\cdot))(\xi_{j,m})\\
	&
	=\frac{1}{\sigma_{k,m}}	\sum_{l=1}^m (T_m)_{j,l} (z_{k,m})_l = (z_{k,m})_j = \sqrt{\frac{2}{m+1}} \sin(k\pi \xi_{j,m}).
\end{align*}
\end{proof}
\section{Proof of claim \ref{claim:sv}}\label{app2}
\begin{proof}[Proof of claim \ref{claim:sv}]
		We need the following auxiliary identity: For $m\in\N, t\in\N_0 $ and $k\in\{1,...,m\}, s\in\{0,...,m\}$ and $j=t(m+1)+s$ there holds
		\begin{align}\label{p1:e1a}
		\sum_{l=1}^m \sin\left(\frac{j \pi l}{m+1}\right) \sin\left(\frac{ k \pi l}{m+1}\right)=\begin{cases} \frac{m+1}{2} &\quad \mbox{for }s=k~\mbox{and}~t~\mbox{even}  \\
			-\frac{m+1}{2} &\quad \mbox{for }s+k=m+1~\mbox{and}~t~\mbox{odd}\\
			0 &\quad \mbox{else}\end{cases}.
	\end{align}
		We first prove the claim. With the polar identity we obtain
	\begin{align*}
		\sum_{l=1}^m \sin\left(\frac{j \pi l}{m+1}\right) \sin\left(\frac{ k \pi l}{m+1}\right)
		= \frac{1}{4}\sum_{l=1}^m\left(q_2^l+q_2^{-l}-(q_1^l+q_1^{-l})\right),
	\end{align*}
    where $q_1=\exp\left(i(j+k)\pi/(m+1)\right)$ and $q_2=\exp\left(i(j-k)\pi/(m+1)\right)$.
	For $q\in\{q_1,q_2\}$, if $q\neq 0,1$, if holds that
	\begin{align*}
		\sum_{i=1}^m (q^i + q^{-i})&= -1 + \frac{q^{m+\frac{1}{2}} - q^{-\left(m+\frac{1}{2}\right)}}{q^\frac{1}{2}-q^{-\frac{1}{2}}}= -1 + (-1)^{k+j}(-1)=-(1+(-1)^{k+j})
				\end{align*}
	since $q^{m+\frac{1}{2}} =(-1)^{k+j}q^{-\frac{1}{2}}$. If $t$ is even and
	$s=k$, then $j-k=t(m+1)$ which implies that $q_2=1$, while, since
	$0<2k<2(m+1)$, the sum $j+k=t(m+1)+2k$ cannot be a multiple of $2(m+1)$,
	therefore $q_1\neq 0,1$ and thus, since $j+k$ is even, $\sum_{l=1}^m
	\sin\left(\frac{j \pi l}{m+1}\right) =\frac{m+1}{2}$. Similarly, if $t$ is odd
	and $s+k=m+1$, then $j+k=(t+1)(m+1)$ implies  $q_1=1$, and now
	$j-k=t(m+1)+s-k=(t+1)(m+1)-2k$ is not a multiple of $2(m+1)$, which yields
	$q_2\neq 0,1$. Since $j+k$ is again even we deduce		$\sum_{l=1}^m
	\sin\left(\frac{j \pi l}{m+1}\right) \sin\left(\frac{ k \pi
	l}{m+1}\right)=-\frac{m+1}{2}$.	In any other case it holds that $q_1,q_2\neq
	0,1$ and therefore $\sum_{l=1}^m \sin\left(\frac{j \pi l}{m+1}\right)
	\sin\left(\frac{ k \pi l}{m+1}\right)=0$, which finishes the proof of the claim
	\eqref{p1:e1a}. We come to the proof of the proposition. As above we can write
	$j=t(m+1)+s$ with $t\in\N_0$ and $s\in\{0,...,m\}$. Using the claim
	\eqref{p1:e1a} together with
	\begin{align*}
		&\sigma_{j,m}\frac{m+1}{2}(v_k,v_{j,m}) = \left(\sin(\sqrt{\lambda_k} \cdot),\sum_{l=1}^m\sin\left(\sqrt{\lambda_j}\xi_l)\kappa(\xi_{l,m},\cdot)\right)\right)\\
		 =~ &\sum_{l=1}^m\sin\left(\sqrt{\lambda_j}\xi_l\right) \left(\sin(\sqrt{\lambda_k}\cdot),\kappa(\xi_{l,m},\cdot)\right)=\sigma_k\sum_{l=1}^m\sin\left(\sqrt{\lambda_j}\xi_l\right)\sin\left(\sqrt{\lambda_k}\xi_l\right)
	\end{align*}
	concludes the proof.
\end{proof}
\bibliographystyle{iopart-num}
\bibliography{references}
\end{document}